\documentclass{article}
\usepackage{amsmath,amsfonts,amsthm,amssymb,mathrsfs,latexsym,paralist, xypic}
\usepackage{tikz, soul}
\usepackage{lipsum}

\usetikzlibrary{arrows,automata}

\tikzstyle{V}=[fill=black,circle,scale=0.2, outer sep = 4pt]

\newtheorem{thm}{Theorem}[section]
\newtheorem{prop}[thm]{Proposition}
\newtheorem{cor}[thm]{Corollary}
\newtheorem{lemma}[thm]{Lemma}
\theoremstyle{remark}
\newtheorem{rmk}[thm]{Remark}
\newtheorem{example}[thm]{Example}
\theoremstyle{definition}
\newtheorem{defn}[thm]{Definition}
\newtheorem{proposition}[thm]{Proposition}

\DeclareMathOperator{\Aut}{Aut}

\newcommand{\z}{^{0}}

\renewcommand{\sp}{\textrm{span}}

\newcommand{\bi}{\begin{itemize}}
\newcommand{\ei}{\end{itemize}}
\newcommand{\be}{\begin{enumerate}}
\newcommand{\ee}{\end{enumerate}}

\newcommand{\C}{\mathbb{C}}

\newcommand{\T}{\mathbb{T}}

\newcommand{\R}{\mathbb{R}}
\newcommand{\N}{\mathbb{N}}
\newcommand{\Z}{\mathbb{Z}}

\renewcommand{\l}{\langle}
\newcommand{\clsp}{\operatorname{\overline{span}}}
\newcommand{\Quad}{\quad\quad\quad\quad\quad\quad\quad\quad\quad}
\newcommand{\qqquad}{\quad\quad\quad\quad\quad}

\renewcommand{\r}{\rangle}

\providecommand{\keywords}[1]{{\textit{Key words and phrases:}} #1}
\providecommand{\classification}[1]{{\textit{2010 Mathematics Subject Classification:}} #1}

\usepackage[margin=1in]{geometry}

\AtEndDocument{\bigskip{\footnotesize%
  \textsc{Carla Farsi, Elizabeth Gillaspy, Judith Packer : Department of Mathematics, University of Colorado at Boulder, Boulder, Colorado, 80309-0395} \par
  \textit{E-mail address}: \texttt{farsi@euclid.colorado.edu, Elizabeth.Gillaspy@colorado.edu, packer@euclid.colorado.edu} \par
  \addvspace{\medskipamount}
  \textsc{Sooran Kang : Department of Mathematics and Statistics, University of Otago, P.O. Box 56, Dunedin 9054, New Zealand.} \par
  \textit{E-mail address}, \texttt{sooran@maths.otago.ac.nz}
}}

\title{Separable representations, KMS states, and wavelets for higher-rank graphs}
\author{Carla Farsi, Elizabeth Gillaspy, Sooran Kang, and Judith Packer}

\begin{document}

\maketitle

\begin{abstract}
Let $\Lambda$ be a strongly connected, finite higher-rank graph.  In this paper, we construct  representations of $C^*(\Lambda)$ on certain separable Hilbert spaces of the form $L^2(X,\mu)$, by introducing the notion of a  $\Lambda$-semibranching function system (a generalization of the  semibranching function systems studied by Marcolli and Paolucci). In particular, if $\Lambda$ is aperiodic, we obtain a faithful representation of $C^*(\Lambda)$ on $L^2(\Lambda^\infty, M)$, where $M$ is the  Perron-Frobenius probability measure on the infinite path space $\Lambda^\infty$ recently studied by an Huef, Laca, Raeburn, and Sims.  We also show how a $\Lambda$-semibranching function system gives rise to KMS states for $C^*(\Lambda)$. For  the higher-rank graphs of Robertson and Steger, we also obtain a   representation of $C^*(\Lambda)$ on  $L^2(X, \mu)$, where $X$ is a fractal subspace of $[0,1]$ by embedding $\Lambda^{\infty}$ into $[0,1]$ as a fractal subset $X$ of  $[0,1]$.
In this latter case  we additionally show that there exists a KMS state for $C^*(\Lambda)$ whose inverse temperature is equal to the Hausdorff dimension of $X$.  Finally, we construct a wavelet system for $L^2(\Lambda^\infty, M)$  by generalizing the work of Marcolli and Paolucci from graphs to higher-rank graphs.
\end{abstract}

\classification{46L05.}


\tableofcontents

\section{Introduction}

 Higher-rank graphs (or $k$-graphs) and their $C^*$-algebras were introduced by Kumjian and Pask in \cite{KP1} as generalizations of Cuntz-Krieger  $C^*$-algebras associated to directed graphs (cf.~\cite{enomoto-watatani-1, bprw, kpr}).  Building on work by Robertson and Steger \cite{RS-London, RS}, the higher-rank graph $C^*$-algebras of \cite{KP1}, and their twisted counterparts (developed in \cite{KPS-hlogy, KPS-twisted, SWW}) 
 share many of the important properties of graph $C^*$-algebras, including  Cuntz-Krieger uniqueness theorems and realizations as groupoid $C^*$-algebras.  Moreover, many important examples of $C^*$-algebras (such as noncommutative tori) can be viewed as twisted $k$-graph  $C^*$-algebras \cite{KPS-hlogy}. Further examples of  $k$-graph $C^*$-algebras, including specific examples, and relationships with dynamical systems theory, can be found in recent work of Pask, Raeburn, and collaborators (\cite{paskraebuweaver-periodic, paskraebuweaver-family}). Over the years, many techniques have been developed for analyzing  $K$-groups of (twisted)  Cuntz-Krieger $k$-graph $C^*$-algebras \cite{RS-K-Th, E, allen-pask-sims-dual}  and their primitive ideal spaces \cite{CKSS, KP}, as well as for studying KMS states associated to a variety of dynamical systems associated on them \cite{aHLRS1, aHLRS2, aHLRS3, aHKR}.

Although several different types of  representations of Cuntz-Krieger $k$-graph $C^*$-algebras have been studied in many of the references cited above, these representations have almost always been on nonseparable $\ell^2$-spaces canonically associated to the underlying  higher-rank graphs. Robertson and Steger  noted in Remark 3.9 of \cite{RS} that there exist nontrivial  representations of their  higher rank Cuntz-Krieger algebras on separable Hilbert spaces, but these representations do not seem to have been explicitly constructed.

One of our main goals in this paper is to describe faithful  separable representations of Cuntz-Krieger $C^*$-algebras associated to strongly connected, finite, aperiodic $k$-graphs   (of which the Robertson-Steger algebras of \cite{RS} are an example): see Theorems \ref{th-branching-sys-gives-repres} and  \ref{prop:M=Hausd}.
Additionally, we use one of these representations to construct a wavelet system on $L^2(\Lambda^\infty, M)$, where $M$ is the Perron-Frobenius Borel probability measure (hereafter referred to as the Perron-Frobenius measure) on the infinite path space $\Lambda^\infty$ constructed in Proposition 8.1 of \cite{aHLRS3}.
We also study the KMS states associated to these   representations.

To construct our representations, as well as the wavelets mentioned above, we build on the work of Marcolli and Paolucci (\cite{MP}) and Bezuglyi and Jorgensen (\cite{bezuglyi-jorgensen}).  Marcolli and Paolucci use the concept of a semibranching function system to define representations of the Cuntz-Krieger $C^*$-algebra $\mathcal{O}_A$ 
on several different Euclidean  fractal spaces associated to the  matrix $A$, while Bezuglyi and Jorgensen construct representations of $\mathcal{O}_A$ on infinite path spaces associated to stationary Bratteli diagrams.
Indeed, many of Marcolli and Paolucci's  representations 
are also on the infinite path space  $(X_A, \mu_A)$ of the 1-graph with adjacency matrix $A$, 
where $\mu_A$ is a probability measure on $X_A$ associated to the Perron-Frobenius eigenvector of $A$.

The existence of similar measure spaces for higher-rank graphs is well established in the literature.
Already  in  \cite{KPActn},  Kumjian and Pask had described a probability measure on the two-sided infinite path space of a $k$-graph for $k>1,$ and in \cite{aHLRS3}, an Huef, Laca, Raeburn and Sims  detail a Perron-Frobenius probability measure on the one-sided infinite path space of a strongly connected finite $k$-graph.  Thus, we hoped that analogues of the semibranching function systems of \cite{MP} and \cite{bezuglyi-jorgensen} would allow us to construct faithful representations of higher-rank graphs on such measure spaces.

In order to construct our  separable representations, we first present a generalization of the notion of 
a semibranching function system that has as part of its associated data a  probability measure $\mu$ on a (fractal) space $X.$   In almost all of our examples $X = \Lambda^{\infty}$ with $\mu= M$ the Perron-Frobenius measure introduced in \cite{aHLRS3} (cf.~Definition \ref{def:perron-frobenius} below). In particular, our generalized ``$\Lambda$-semibranching function systems" 
are systems of  partially defined shift operators  for which the associated Radon-Nikodym derivatives are almost everywhere non-zero  on their domains (see Definition \ref{def-lambda-brach-system} for details).  From a  $\Lambda$-semibranching function system, Theorem \ref{th-branching-sys-gives-repres} then tells us how to obtain a  representation of $C^*(\Lambda)$ on  $L^2(X, \mu)$, which is faithful when $\Lambda$ is aperiodic.   We give two examples of these representations in Proposition \ref{prop-main-exampl} (on $L^2(\Lambda^\infty, M)$) and in Theorem \ref{prop:M=Hausd} (on a fractal subspace $X$ of $[0,1]$).  We construct the space $X$ by using the vertex adjacency matrices  $A_1, \ldots, A_k$ of $\Lambda$ and thinking of points in $[0,1]$ in their $N$-adic expansions, where $N=|\Lambda^0|$.

In the case $k=1$, our results provide a complementary perspective to those of Marcolli and Paolucci in \cite{MP}, since we do not require (as they do) that the vertex adjacency matrices have entries from $\{0,1\}$.  This can be explained by their interpretation of  the Cuntz-Krieger algebra $\mathcal{O}_A$ as being associated to the adjacency matrix $A$ indexed by the \emph{edges} of a directed graph; whereas we study the vertex adjacency matrices $A_i$ indexed by the \emph{vertices} of a graph or $k$-graph.

 We note that $\Lambda$-semibranching function systems  also  provide a template for establishing the existence of faithful representations of $C^*(\Lambda)$ on other Hilbert spaces.
Indeed, examining recent work of Bezuglyi and Jorgensen  \cite{bezuglyi-jorgensen}, and Jorgensen and Dutkay \cite{dutkay-jorgensen-monic, dutkay-jorgensen-markovmeasure}, we conjecture that the Perron-Frobenius measure of Definition \ref{def:perron-frobenius}, although a useful and canonical example of a probability measure, is potentially just one of many measures that could give faithful representations of  Cuntz-Krieger  $C^*$-algebras associated to  strongly connected finite $k$-graphs.

%
%
In future work \cite{farsi-gillaspy-kang-packer}, we hope to classify $\Lambda$--monic and atomic representations of  Cuntz-Krieger  $C^*$-algebras associated to  strongly connected $k$-graphs  and higher-rank Bratteli diagrams thus generalizing some of the main  results in   \cite{bezuglyi-jorgensen,dutkay-jorgensen-monic,dutkay-jorgensen-markovmeasure,dutkay-jorgensen-atomic}.

The $\Lambda$-semibranching function systems that we construct can also be used to give an explicit construction of many of the KMS states on $C^*(\Lambda)$ whose existence was established by an Huef, Laca, Raeburn, and Sims in \cite{aHLRS3}.

The definition of a KMS state arises from physics.  In this context, a KMS state on the $C^*$-algebra of observables of a physical model represents an equilibrium state with respect to a time evolution (represented by an action of $\R$).  KMS states can be characterized by a commutation condition, which makes sense for any $C^*$-algebra $A$.  Recent research (cf.~\cite{bost-connes, exel, aHLRS1, aHLRS2, laca-etc, aHLRS3}) into the KMS states of abstract $C^*$-algebras has shown that the KMS states of a $C^*$-algebra $A$ often encode  information about important structural properties of $A$.

In \cite{aHLRS3}, an Huef, Laca, Raeburn, and Sims provide a complete description of the KMS states of $C^*(\Lambda)$ for  a strongly connected finite $k$-graph $\Lambda$.  Their description relies on representations of the periodicity group of $\Lambda$. In Section \ref{sec:kms}, we give an explicit construction of many of these KMS states, by using $\Lambda$-semibranching function systems instead of the periodicity group. 
 To be precise, Corollary \ref{cor:KMS} provides an alternative proof  of part of Theorem 11.1 of \cite{aHLRS3}.  Moreover, when  $\Lambda$ is a Robertson-Steger $k$-graph in the sense that the vertex matrices $A_i$ of $\Lambda$ have only $\{0,1\}$ entries and $A_1\dots A_k$ has also $\{0,1\}$ entries, we construct in Corollary \ref{cor-temp-eq-haus-measu}  a KMS$_\beta$ state on $C^*(\Lambda)$ whose inverse temperature $\beta$ is equal to the the fractal dimension of the subspace $X$ of $[0,1]$ arising from the $N$-adic representations of infinite paths of $\Lambda$.

 Having established the existence of  separable representations of $C^*(\Lambda)$ for strongly connected finite $k$-graphs $\Lambda$, we show in Section \ref{sec:wavelets} how to obtain a wavelet-type decomposition of the Hilbert space $L^2(X, \mu)$ associated to our $\Lambda$-semibranching function systems.   As in the work of Jonsson (\cite{jonsson}) and Marcolli and Paolucci (\cite{MP}), one motive for constructing these wavelets is that they give methods of constructing different function spaces on the infinite path spaces for the $k$-graphs being studied, and illustrate how representations of the $C^{\ast}$-algebras corresponding to higher-rank graphs can in turn give information about the path spaces of those graphs.

Additionally, as explained in \cite{crovella-kolaczyk}, a wavelet decomposition of a Hilbert space associated to a network or directed graph is quite useful for performing spatial traffic analysis on the network.  Higher-rank graphs can be viewed as quotients of edge-colored directed graphs,
and in future work we hope to extend the results of Section 5.2 of \cite{MP} to $\Lambda$-semibranching function systems, which we hope will prove useful for spatial traffic analysis on networks with qualitatively different edges.

To describe in further detail our results on wavelets associated to higher-rank graphs, we first recall that Marcolli and Paolucci,  inspired in part by the work of Jonsson in \cite{jonsson}, constructed in \cite{MP} two different families of wavelets associated to the Cuntz-Krieger $C^*$-algebra $\mathcal{O}_A$.    In both \cite{MP} and \cite{jonsson}, the authors provide families of wavelets in $L^2(\Lambda_A,\mu),$ where $\Lambda_A$ is the $N$-adic fractal associated to the matrix $A$, and $\mu$ is the associated Hausdorff measure on $\Lambda_A$. One of the constructions of Marcolli and Paolucci uses the Perron-Frobenius theory of irreducible matrices, and provides a finite set of functions in $L^2(\Lambda_A,\mu)$ that can be shifted around by the generating isometries in $\mathcal{O}_A$ to provide an orthonormal basis for the orthogonal complement of a specified initial space $V_0.$

Building on this construction,  we are able to show that our examples of $\Lambda$-semibranching function systems give rise to a wavelet system in $L^2(\Lambda^{\infty}, M)$, where $M$ is the Perron-Frobenius measure on $\Lambda^\infty$.  
As in \cite{MP}, rather than being dilated and translated, the wavelets are shifted by the partial isometries on $L^2(\Lambda^\infty, M)$ which generate the separable representation of $C^*(\Lambda)$;
see Theorem \ref{Wavelets-Theo} for details.

\subsection{Structure of the paper}
We begin in Section \ref{sec:background} by reviewing the basic concepts we will rely on throughout this paper: (strongly connected) higher-rank graphs, their associated Cuntz-Krieger $C^*$-algebras, and their KMS states. We also present  several useful facts from  Perron-Frobenius theory. 
In this section we also quickly review the basic properties of Hausdorff measures, as we will need these for one of our examples in Section \ref{sec:example}.  In Section \ref{sec:sbfs} we present our definition of  a $\Lambda$-semibranching function system on a finite $k$-graph, and establish our main result (Theorem \ref{th-branching-sys-gives-repres}), namely, that such a $\Lambda$-semibranching function system always gives rise to a  representation of $C^*(\Lambda)$ on a separable Hilbert space $L^2(X, \mu)$, which is faithful when $\Lambda$ is aperiodic.  We also present in Proposition \ref{prop-main-exampl} our main example of $\Lambda$-semibranching function systems, together with a fractal interpretation of it, see Theorem  \ref{prop:M=Hausd}.   We describe how these $\Lambda$-semibranching function systems give rise to KMS states on $C^*(\Lambda)$ in Section \ref{sec:kms}, and show in Section \ref{sec:wavelets} how they give us a wavelet decomposition of $L^2(\Lambda^\infty, M)$ (see Theorem \ref{Wavelets-Theo}).

\subsection{Acknowledgments}  This work was partially supported by a grant from the Simons Foundation (\#316981 to Judith Packer).

\section{Background}
\label{sec:background}

\subsection{Higher-rank graphs and their vertex matrices.} Let $k\in \N$ with $k\ge 1$. We write $e_1,\dots, e_k$ for the generators of $\N^k$. A higher-rank graph, or $k$-graph, is a countable category $\Lambda$ equipped with a functor $d:\Lambda\to \N^k$ satisfying the \textit{factorization property}: for every morphism $\lambda\in\Lambda$ and $m,n\in\N^k$ with $d(\lambda)=m+n$, there exist unique morphisms $\mu,\nu\in\Lambda$ such that $\lambda=\mu\nu$ and $d(\mu)=m$, $d(\nu)=n$.  We will often call morphisms $\lambda \in \Lambda$ \emph{elements} or \emph{(finite) paths} in $\Lambda$, in keeping with our understanding of $k$-graphs as higher-dimensional generalizations of directed graphs. The elements in $\Lambda^0$ are the identity morphisms, and we call them vertices. We write $r,s:\Lambda\to \Lambda^0$ for  the range and source maps in $\Lambda$. For $v,w\in\Lambda^0$ and $n\in\N^k$, we write
\[
\Lambda^n:=\{\lambda\in\Lambda:  d(\lambda)=n\}\;\;\text{and}\;\;v\Lambda w:=\{\lambda\in\Lambda: r(\lambda)=v, s(\lambda)=w\}.
\]
Also for $\mu,\nu\in\Lambda$, we write
\[
\Lambda^{\min}(\mu,\nu)=\{(\eta,\zeta)\in\Lambda\times\Lambda:\mu\eta=\nu\zeta, d(\mu\eta)=d(\mu)\vee d(\nu)\}.
\]
We say that $\Lambda$ is \textit{finite} if $\Lambda^n$ is finite for all $n\in\N^k$ and say that $\Lambda$ has \textit{no sources} if $v\Lambda^n\ne \emptyset$ for all $v\in\Lambda^0$ and $n\in\N^k$; this is equivalent to saying that $v\Lambda^{e_i}\ne \emptyset$ for all $v\in\Lambda$ and all $e_i$. 

\begin{defn}
\label{def:strongly-conn}
\begin{enumerate}
\item  We say that $\Lambda$ is \textit{strongly connected} if, for all $v,w\in\Lambda^0$,  $v\Lambda w\ne \emptyset$.
\item For $1\le i\le k$, let $A_i$ be the matrix of $M_{\Lambda^0}(\N)$ with entries $A_i(v,w)=\vert v\Lambda^{e_i}w\vert$, the number of paths from $w$ to $v$ with degree $e_i$; we call the $A_i$  the \textit{vertex adjacency matrices} of $\Lambda$, or more simply the \textit{``vertex matrices''} for $\Lambda$.
\end{enumerate}
\end{defn}

 We note that if $\Lambda$ is strongly connected, then $\Lambda$ has no sources by Lemma~2.1 of \cite{aHLRS3}.  Also, the factorization property of $\Lambda$ implies that $A_iA_j=A_jA_i$.

 The following definition comes from \cite{aHLRS3} Section 3. 

\begin{defn}
\label{def:irreduc-family-def}
(\cite{aHLRS3} Section 3)
Let $\mathcal{A}= \{A_1, ..., A_k\}$ be a family of nonzero commuting $N \times N$ matrices. We say that  $\mathcal{A}= \{A_1, ..., A_k\}$ is \emph{irreducible} if for every $(s,t )\in N^2$ there exists $ n = (n_1, \ldots, n_k) \in \N^k$ 
(depending on  $(s,t )$) such that
\[
A_1^{n_1}....A_k^{n_k}(s,t) >0.
\]
\end{defn}

To justify the next definition, note  that if  the family of matrices $\mathcal{A}=\{A_1, \ldots, A_k\}$ is irreducible in the sense of Definition \ref{def:irreduc-family-def}, then  the proof of Proposition 3.1 of \cite{aHLRS3} implies that the matrices $A_i$ have a unique common unimodular (i.e., of  $\ell^1$-norm one) eigenvector  with positive entries $x^\Lambda$, such that
\[ A_i x^\Lambda = \rho(A_i) x^\Lambda.\]
Here $\rho(A_i)$ denotes the spectral radius of $A_i$.

It follows from the fact that commuting matrices have the same eigenspaces that $x^\Lambda$ is also the unique unimodular eigenvector with positive entries for the product matrix $A_1 \cdots A_k$.

\begin{defn}
\label{def:irreduc-family-perr-frob-def} Let $\Lambda$ be a strongly connected $k$-graph with vertex matrices $A_1, \ldots, A_k$.
\begin{enumerate}
\item We write $x^\Lambda$ for the unique common unimodular Perron-Frobenius eigenvector of the vertex matrices $A_i$. Note that in particular $x^\Lambda$  has all of its entries positive.  We will call $x^\Lambda$ the \emph{Perron-Frobenius eigenvector} of $\Lambda$.
\item We call eigenvalues associated to $x^\Lambda$  the \textit{Perron-Frobenius eigenvalues}.  As observed above, the Perron-Frobenius eigenvalue for $A_i$ is the spectral radius $\rho(A_i)$ of $A_i$.
\end{enumerate}
\end{defn}

\begin{proposition}  (Lemma~4.1 of \cite{aHLRS3})
\label{def:strongly-conn-char}
A finite $k$-graph $\Lambda$ with vertex matrices $A_1, ..., A_k$ is strongly connected if and only if
$\mathcal{A}= \{A_1, ..., A_k\}$ is an irreducible family of matrices.
\end{proposition}

Also note that Definition \ref{def:irreduc-family-def} of an irreducible family  is the same as  the definition of irreducibility given  by Robertson and Steger (c.f. \cite{RS-Irr} page 94).

\subsection{Infinite path spaces and probability measures}
Let $\Lambda$ be a finite $k$-graph with no sources. To discuss the \textit{infinite path space} $\Lambda^\infty$, consider the set
\[
\Omega_k:=\{(p,q)\in\N^k\times\N^k:p\le q\}.
\]
We make $\Omega_k$ into a $k$-graph as follows.  Let $\Omega_k^0 = \N^k$, and define $r,s:\Omega_k\to \N^k$ by $r(p,q):=p$ and $s(p,q):=q$. We define composition by $(p,q)(q,m)=(p,m)$ and degree by $d(p,q)=q-p$. Then $\Omega_k$ is a $k$-graph with no sources. As in Definition~2.1 of \cite{KP1}, an \textit{infinite path} in a $k$-graph $\Lambda$ is a $k$-graph morphism $x:\Omega_k\to \Lambda$. We write $\Lambda^\infty$ for the collection of all infinite paths and call it the \textit{infinite path space} of $\Lambda$. For each $p\in \N^k$, we define $\sigma^p:\Lambda^\infty \to \Lambda^\infty$ by $\sigma^p(x)(m,n)=x(m+p,n+p)$ for $x\in\Lambda^\infty$. For $\lambda\in\Lambda$ we define $Z(\lambda)=\{x\in\Lambda^\infty:x(0,d(\lambda))=\lambda\}$ and we call it a \textit{cylinder set}. It is shown in \cite{KP1} that the cylinder sets $\{Z(\lambda)\}$ are a basis for the topology on $\Lambda^\infty$. Note that $\Lambda^\infty$ is compact if and only if $\Lambda^0$ is finite.

We say that a $k$-graph $\Lambda$ is \emph{aperiodic} if for each $v \in \Lambda^0$, there exists $x \in Z(v)$ such that for all
$m \not= n \in \N^k$, we have $\sigma^m(x) \not= \sigma^n(x)$. 

The following definition can be found originally in Proposition~8.1 of \cite{aHLRS3}.  \begin{defn}
\label{def:perron-frobenius}
 Let $\Lambda$ be a strongly connected finite $k$-graph with the vertex matrices $A_i$.  Define a measure $M$ on $\Lambda^\infty$ by
\begin{equation}\label{eq:measure}
M(Z(\lambda))=\rho(\Lambda)^{-d(\lambda)}x^\Lambda_{s(\lambda)}\quad\text{for all}\;\lambda\in\Lambda,
\end{equation}
where  $\rho(\Lambda)=(\rho(A_1),\dots,\rho(A_k))$ and $x^\Lambda$ is the  unimodular Perron-Frobenius eigenvector of $\Lambda$ of Definition
\ref{def:irreduc-family-perr-frob-def}.  We will call $M$ the \emph{Perron-Frobenius measure} on $\Lambda^\infty$.
\end{defn}

Proposition~8.1 of \cite{aHLRS3} establishes that $M$ is the unique Borel probability measure on $\Lambda^\infty$ that satisfies
\[
M(Z(\lambda))=\rho(\Lambda)^{-d(\lambda)}M(Z(s(\lambda))\quad\text{for all}\;\;\lambda\in\Lambda.
\]

\subsection{Cuntz-Krieger  $C^*$-algebras of $k$-graphs}
\begin{defn}
 Let $\Lambda$ be a finite $k$-graph with no sources. A \emph{Cuntz-Kriger $\Lambda$-family} is a collection $\{t_\lambda:\lambda\in\Lambda\}$ of partial isometries in a $C^*$-algebra $A$ such that
\begin{itemize}
\item[(CK1)] $\{t_v: v \in \Lambda\z\}$ is a family of mutually orthogonal projections,
\item[(CK2)] $t_\mu t_\lambda = t_{\mu \lambda}$ whenever $s(\mu) = r(\lambda)$,
\item[(CK3)] $t_\mu^* t_\mu = t_{s(\mu)}$ for all $\mu$, and
\item[(CK4)] for all $v\in\Lambda^0$ and $n\in\N^k$, we have
\[
t_v=\sum_{\lambda\in v\Lambda^n} t_\lambda t^*_\lambda.
\]
\end{itemize}
These relations imply that for all $\mu,\nu\in\Lambda$
\begin{equation}\label{eq:MCE}
t^*_\mu t_\nu =\sum_{(\eta,\zeta)\in\Lambda^{\min}(\mu,\nu)}t_\eta t^*_\zeta,
\end{equation}
where we interpret empty sums as zero.

The \emph{Cuntz-Krieger  $C^*$-algebra} $C^*(\Lambda)$ associated to  $\Lambda$ is generated by a universal Cuntz-Krieger $\Lambda$-family $\{s_\lambda:\lambda\in\Lambda\}$, and we can show that
\[
C^*(\Lambda)=\clsp\{s_\mu s^*_\nu:\mu,\nu\in\Lambda,s(\mu)=s(\nu)\}.
\]
The universal property gives a gauge action $\gamma$ of $\T^k$ on $C^*(\Lambda)$ such that $\gamma_z(s_\lambda)=z^{d(\lambda)}s_\lambda$, where $z^n=\prod_{i=1}^k z^{n_i}_i$ for $z=(z_1,\dots,z_k)\in \T^k$ and $n\in \Z^k$.
\end{defn}

\subsection{Dynamics and KMS states on $C^*(\Lambda)$}
\label{subsection:background_Dyn}

Suppose $(A,\alpha)$ is a dynamical system consisting of an action $\alpha$ of $\R$ on a $C^*$-algebra $A$. We say that $a\in A$ is \textit{analytic} for $\alpha$ if the function $t\mapsto \alpha_t(a)$ is the restriction to $\R$ of an analytic function $z\mapsto \alpha_z(a)$ defined on $\C$. A state $\phi$ on $A$ is a \textit{KMS state at the inverse temperature $\beta$} (or a \textit{KMS$_\beta$ state} of the system $(A,\alpha)$) if
\begin{equation}\label{eq:KMS_cond}
\phi(ab)=\phi(b\alpha_{i\beta}(a))
\end{equation}
for all analytic elements $a,b$. According to Proposition~8.12.3. of \cite{Ped}, it suffices to check the KMS condition on a set of analytic elements which span a dense subspace of $A$.

Let $\Lambda$ be a finite $k$-graph with no sources with the vertex matrices $A_i$.
Let $r\in (0,\infty)^k$ and define $\alpha^r:\R \to \Aut(C^*(\Lambda))$ in terms of the gauge action by $\alpha^r=\gamma_{e^{itr}}$. Then for $\mu\in\Lambda$, we have
\[
\alpha^r_t(t_\mu t^*_\nu)=e^{itr\cdot(d(\mu)-d(\nu))}t_\mu t^*_\nu
\]
is the restriction of the analytic function $z\mapsto e^{izr\cdot(d(\mu)-d(\nu))} t_\mu t^*_\nu$. Thus it suffices to check the KMS condition \eqref{eq:KMS_cond} on the elements $t_\mu t^*_\nu$.

We are particularly interested in the dynamics $\alpha^r$ with
\[
r=\ln \rho(\Lambda)=(\ln \rho(A_1),\dots,\ln \rho(A_k))\in (0,\infty)^k
\]
on $C^*(\Lambda)$; this dynamics is called the \textit{preferred dynamics}.

\subsection{Hausdorff measure and Hausdorff dimension}
Here, we review the definition of Hausdorff outer measure and Hausdorff dimension for subsets of $\mathbb R^n.$  More details and proofs of these facts can be found in the book by K. Falconer \cite{falconer}, or Chapter 6 of the book by G. Edgar \cite{edg}, for example.
\begin{defn}
Fix $n, s, \delta$ with $n\in\mathbb N,\;0\leq s<\infty,$ and $\delta>0.$  For each subset $E\subset \mathbb R^n,$ let
$$H^s_{\delta}(E)=\text{inf}\{\sum_{j\geq 1}[\text{diam}(A_j)]^s\},$$
where the infimum is taken over all countable collections of
subsets $\{A_j\}_{j\geq 1}$ of $\mathbb R^n$ such that
$$E\subset \cup_{j\geq 1}A_j\;\;\text{and}\;\;\text{diam}(A_j)<\delta,\forall j\geq 1.$$
This function $H^s_{\delta}$ is called the {\it Hausdorff outer measure} on $\mathbb R^n$ associated to $s$ and $\delta.$
\end{defn}
One sees that for $E\subset \mathbb R^n,\;s\geq 0$ fixed, and $0<\delta_1<\delta_2$ we have
$$H^s_{\delta_2}(E)\leq H^s_{\delta_1}(E)$$ since one is taking the infimum over a larger family of coverings for larger $\delta.$
\begin{defn}
For $0\leq s<\infty,$ and $E\subset \mathbb R^n,$ define
$$H^s(E)=\lim_{\delta\to 0+}H^s_{\delta}(E)=\text{sup}\{H^s_{\delta}(E):\delta>0\}.$$
One verifies the Carath\'eodory criterion, i.e. if $E_1$ and $E_2$ are subsets of $\mathbb R^n$ with
$$d(E_1,E_2)=\text{inf}\{\|x-y\|:\;x\in E_1,\;y\in E_2\}=\epsilon>0,$$
then $$H^s(E_1\cup E_2)\;=H^s(E_1)+H^s(E_2),$$
so that $H^s$ is also an outer measure on $\mathbb R^n,$ called the \emph{Hausdorff outer measure} associated to $s\geq 0.$
\end{defn}
\begin{thm}
\label{thm-Hausd-measure}
Fix $s\in [0,\infty).$ Define the family of measurable sets for
$H^s,$ denoted by ${\mathcal M}(H^s),$ by
$${\mathcal M}(H^s)\;=\;\{E\subset \mathbb R^n| H^s(A)=H^s(A\cap E)+H^s(A-E),\;\forall A\subset \mathbb R^n\}.$$ Then:
\begin{enumerate}
\item[(i)] The Borel sets in $\mathbb R^n,\;{\mathcal B}(\mathbb R^n),$ satisfy ${\mathcal B}(\mathbb R^n)\subset {\mathcal M}(H^s).$
\item[(ii)] If $E\in {\mathcal M}(H^s)$ and $H^s(E)<\infty,$ then there exists an $F_{\sigma}$-set $F\subset \mathbb R^n$ such that $F\subset E$ and $H^s(E-F)=0.$
\item[(iii)] For $E\in{\mathcal M}(H^s),$ there exists a $G_{\delta}$-set $G$ with $E\subset G$ and $H^s(E)=H^s(G).$
\item[(iii)] For all $t\in \mathbb R^n,$ and $E\in {\mathcal M}(H^s),$
$$H^s(E)=H^s(E+t).$$
\item[(iv)] If $d>0$ is a positive scalar, and $E\in {\mathcal M}(H^s),$  then $H^s(d\cdot E))=d^{sn}H^s(E).$
\end{enumerate}
\label{thm:hausd-meas-properties}
\end{thm}
The following proposition is a consequence of the above theorem.
\begin{prop}
\label{2.4}
Fix $E\subset \mathbb R^n$ and define the outer measure $H^s$ for $s\in [0,\infty)$ as above. Then:
\begin{enumerate}
\item[(a)] $H^s(E)=0$ for all $s>n.$
\item[(b)] If $H^s(E)<\infty$ for some $s\in [0,\infty),$ then $H^t(E)=0$ for all $t>s.$
\item[(c)] If $H^s(E)>0$ for some $s>0,$ then $H^t(E)=\infty$ for all $t\in [0, s).$
\end{enumerate}
\end{prop}

It follows from the above proposition that for fixed $E\subset \mathbb R^n,$ if there exists $s_0\in [0,\infty )$ such that
$$0<H^{s_0}(E)<\infty,$$
then, for all $t<s_0,\; H^t(E)=\infty,$ and for all $t>s_0,\;H^t(E)=0.$ We note by the proposition above if such a number $s_0$ exists for $E,$ we must have $s_0\leq n.$

\begin{defn}
Let $E\subset \mathbb R^n.$  We say that $E$ has {\it Hausdorff dimension} $s_0\in [0,\infty)$ if $E\in {\mathcal M}(H^{s_0})$ and
$$0<H^{s_0}(E)<\infty.$$
The measure $H^{s_0}$ restricted to subsets of $E$ that are contained in ${\mathcal M}(H^{s_0})$ is called the {\it Hausdorff measure} on $E.$
\end{defn}

\begin{example}
Fix $n,\;d\in\mathbb N,\;d\geq 2.$ Let $B\subseteq \{0,1,\cdots, d-1\}^n$ and for $i\in B$ define $\tau_i:[0,1]^n\to [0,1]^n$ by $\tau_i(x)=\frac{x+i}{d},$ where vectors are added component-wise.
It is well-known that there is a unique compact set ${\mathcal F}_{d,B}\subset [0,1]^n$ satisfying
$${\mathcal F}_{d,B}\;=\;\bigsqcup_{i\in B}\tau_i({\mathcal F}_{d,B}).$$
The Hausdorff dimension of ${\mathcal F}_{d,B}$ is $\frac{\log{|B|}}{\log{d}}.$  If $B=\{0,1,\cdots, d-1\}^n,$ the set ${\mathcal F}_{d,B}$ is just the unit cube $[0,1]^n$ with dimension $n.$
If $n=1,\;d=3,$ and $B=\{0,2\}\subseteq \{0,1,2\},$ this construction gives the standard Cantor set.

\end{example}

\section{$\Lambda$-semibranching function systems and  representations of $C^*(\Lambda)$}
In this section, we show how to construct a  representation of $C^*(\Lambda)$ out of a generalized semibranching function system, which we call a $\Lambda$-semibranching function system (see Definition \ref{def-lambda-brach-system}).  In particular, when $\Lambda$ is a strongly connected, aperiodic, finite $k$-graph, this construction enables us to represent $C^*(\Lambda)$ faithfully on $L^2(\Lambda^\infty, M)$ where $M$ is the Perron-Frobenius measure of Definition \ref{def:perron-frobenius}.

\subsection{$\Lambda$-semibranching function systems}
\label{sec:sbfs}

We begin by recalling from \cite{MP} the definition of a semibranching function system. See also  \cite{bezuglyi-jorgensen}.

\begin{defn}
\label{def-1-brach-system}\cite[Definition~2.1]{MP}\label{defn:sbfs}
Let $(X,\mu)$ be a measure space, and let $I$ be a finite index set such that $\vert I\vert=N$. Suppose that, for each $i \in I$, we have a measurable map $\sigma_i:D_i\to X$, for some measurable subsets $D_i\subset X$. The family $\{\sigma_i\}$ is a \emph{semibranching function system} if the following holds.
\begin{itemize}
\item[(a)] There exists a corresponding family $\{R_i\}_{i=1}^N$ of subsets of $X$ with the property that
\[
\mu(X\setminus \cup_i R_i)=0,\quad\quad\mu(R_i\cap R_j)=0\;\;\text{for $i\ne j$},
\]
where $R_i=\sigma_i(D_i)$.
\item[(b)] There is a Radon-Nikodym derivative
\[
\Phi_{\sigma_i}=\frac{d(\mu\circ\sigma_i)}{d\mu}
\]
with $\Phi_{\sigma_i}>0$, $\mu$-almost everywhere on $D_i$.
\end{itemize}
A measurable map $\sigma:X\to X$ is called a \emph{coding map} for the family $\{\sigma_i\}$ if $\sigma\circ\sigma_i(x)=x$ for all $x\in D_i$.
\end{defn}

The fact that $\sigma \circ \sigma_i = id$ save on a subset of $D_i$ of measure zero implies that we also have $\sigma_i \circ \sigma = id|_{R_i}$ off a set of measure 0.  To see this, let $y \in R_i$ be arbitrary, and write $y = \sigma_i(x)$ for some $x \in D_i$.  Then, unless $x$ is in the set $M$ of measure 0 on which $\sigma \circ \sigma_i \not= id$, we have $\sigma(y) = x$, and thus
\[\sigma_i \circ \sigma(y) = \sigma_i (x) = y.\]
Since $\mu(M) = 0$ and $\sigma_i$ is measurable, $\mu(\sigma_i(M)) = 0$, so the above equality holds for almost all $y \in R_i$.

In \cite{MP}, the authors show how a finite  directed graph associated to an irreducible $\{0,1 \}$-matrix  gives rise to a semibranching function system and thence to a representation of  the  Cuntz-Krieger  $C^*$-algebra associated to the directed graph.  Additionally, in \cite{bezuglyi-jorgensen} the authors consider monic representations of Cuntz-Krieger  $C^*$-algebras associated to semibranching function systems, and their equivalence classes. Our goal in this paper is to generalize some of these constructions to obtain representations of finite higher-rank graph Cuntz-Krieger $C^*$-algebras; to do this we need the following generalization of Definition \ref{defn:sbfs}.

\begin{defn}
\label{def-lambda-brach-system}
Let $\Lambda$ be a finite $k$-graph and let $(X, \mu)$ be a measure space.  A \emph{$\Lambda$-semibranching function system} on $(X, \mu)$ is a collection $\{D_\lambda\}_{\lambda \in \Lambda}$ of measurable subsets of $X$, together with a family of \emph{prefixing maps} $\{\tau_\lambda: D_\lambda \to X\}_{\lambda \in \Lambda}$, and a family of coding maps $\{\tau^m: X \to X\}_{m \in \N^k}$, such that
\begin{itemize}
\item[(a)] For each $m \in \N^k$, the family $\{\tau_\lambda: d(\lambda) = m\}$ is a semibranching function system, with coding map $\tau^m$.
\item[(b)] If $ v \in \Lambda^0$, then  $\tau_v = id$,  and $\mu(D_v) > 0$.
\item[(c)] Let $R_\lambda = \tau_\lambda D_\lambda$. For each $\lambda \in \Lambda, \nu \in s(\lambda)\Lambda$, we have $R_\nu \subseteq D_\lambda$ (up to a set of measure 0), and
\[\tau_{\lambda} \tau_\nu = \tau_{\lambda \nu}\text{ a.e.}\]
 (Note that this implies that up to a set of measure 0, $D_{\lambda \nu} = D_\nu$ whenever $s(\lambda) = r(\nu)$).
\item[(d)] The coding maps satisfy $\tau^m \circ \tau^n = \tau^{m+n}$ for any $m, n \in \N^k$.  (Note that this implies that the coding maps pairwise commute.)
\end{itemize}
\end{defn}

\begin{rmk}
Note that Condition (a) of Definition \ref{def-lambda-brach-system} implies that $\Lambda$ must be a finite $k$-graph, since the requirement that $\{\tau_\lambda: d(\lambda) = m\}$ forms a  semibranching function system implies that $|\Lambda^m|$ is finite, for each $m \in \N^k$.

We also observe that if $k=1$, Definition \ref{def-lambda-brach-system} is a stronger requirement than Definition \ref{def-1-brach-system}. In particular, if $(X, \mu)$ admits a semibranching function system where the index set $I$ corresponds to the edges of a directed graph $\Lambda$, there is no obvious way to define a prefixing map $\tau_v$ for a vertex $v$ of $\Lambda$; and if we instead use the vertices of $\Lambda$ as our index set $I$ for a semibranching function system, as in \cite{MP}, we need not have $\sigma_v = id$ for each vertex $v$ of $\Lambda$.
\end{rmk}

\begin{prop}
\label{prop-main-exampl}
Let $\Lambda$ be a strongly connected finite $k$-graph.  The measure space $( \Lambda^\infty, M)$, together  with the prefixing maps $\{\sigma_\lambda: Z(s(\lambda)) \to Z(\lambda)\}_{\lambda \in \Lambda}$ given by $\sigma_\lambda(x) = \lambda x$ and the coding maps $\{\sigma^m\}_{m \in \N^k}$ given by $
\sigma^m(x)=x(m,\infty)$, forms  a $\Lambda$-semibranching function system.
\label{prop:kgraph}
\end{prop}
\begin{proof}
We first note that if $m, n \in \N^k$ and $x \in \Lambda^\infty$, applying the factorization property to $x(0, m+n)$ tells us that $\sigma^m \circ \sigma^n = \sigma^{m+n}$.  Thus, Condition (d) of Definition \ref{def-lambda-brach-system} holds. To see Condition (b), note that if $v \in \Lambda^0$ then $\sigma_v = id|_{Z(v)}$; and $M(Z(v)) = x^\Lambda_v > 0$ for all $v \in \Lambda^0$. For (c), recall that if $v \in \Lambda^0$, then for any $n \in \N^k$, we have
\[Z(v) = \cup_{\lambda \in v\Lambda^n} Z(\lambda),\]
so $Z(\lambda) = R_\lambda \subseteq Z(v)$.  If $s(\nu) = r(\lambda) = v$, then $D_\nu = Z(v)$, so $R_\lambda \subseteq D_\nu$ as desired  whenever $s(\nu) = r(\lambda)$. Similarly, the factorization property implies that if $s(\nu) = r(\lambda)$ then $\sigma_\nu \circ \sigma_\lambda = \sigma_{\nu \lambda}$.

It only remains to check Condition (a), namely, that for each $m \in \N^k$, the family $\{\sigma_\lambda\}_{d(\lambda) = m}$ forms a semibranching function system on $(\Lambda^\infty, M)$.  Observe that  for fixed $m \in \N^k$, $\cup_{\lambda \in \Lambda^m} Z(\lambda) = \Lambda^\infty$: if $x \in \Lambda^\infty$ then $x \in Z(x(0, m)),$ and $x(0,m) = \lambda$ for a unique $\lambda \in \Lambda^m$.  Thus,
\[M\left(\Lambda^\infty \backslash \cup_{\lambda \in \Lambda^m} \sigma_\lambda(D_\lambda) \right) = M(\emptyset) = 0,\]
so Condition (a) of Definition \ref{defn:sbfs} is satisfied.

Since $\Lambda$ is finite, for any $m\in\N^k$ the set of paths of degree $m$ is finite; since $\Lambda$ is source-free the set is nonempty. Note that if $d(\lambda) = m$, then $\sigma^m \circ \sigma_\lambda = id|_{Z(s(\lambda))}$, so $\sigma^m$ is indeed a coding map for $\{\sigma_\lambda\}_{d(\lambda) = m}$.

Thus, it merely remains to check that the Radon-Nikodym derivative $\Phi_\lambda = \frac{d(M \circ \sigma^m)}{dM}$ is strictly positive almost everywhere on $D_\lambda = Z(s(\lambda))$.

Let $\lambda, \nu \in \Lambda$. We compute
\begin{align*}
M\circ\sigma_\lambda(Z(\nu))& =M(Z(\lambda\nu))=\delta_{s(\lambda), r(\nu)}\rho(\Lambda)^{-d(\lambda\nu)}x^\Lambda_{s(\nu)} \\
 \text{ and }
M(Z(\nu))& =\rho(\Lambda)^{-d(\nu)}x^\Lambda_{s(\nu)}.
\end{align*}
Thus,
\[
\frac{M\circ\sigma_\lambda(Z(\nu))}{M(Z(\nu))}=\delta_{s(\lambda), r(\nu)}\frac{\rho(\Lambda)^{-d(\lambda\nu)}x^\Lambda_{s(\nu)}}{\rho(\Lambda)^{-d(\nu)}x^\Lambda_{s(\nu)}}
=\delta_{s(\lambda), r(\nu)}\rho(\Lambda)^{-d(\lambda)}.
\]
Since this value is the same for any $\nu \in s(\lambda)\Lambda$, it follows that $\Phi_\lambda$ is constant on $D_\lambda$: for any $x \in Z(s(\lambda))$ we have
\[\Phi_\lambda(x) = \rho(\Lambda)^{-d(\lambda)} > 0\]
because
  $\rho(\Lambda) \in (0, \infty)^{k}$.  Thus, we have a $\Lambda$-semibranching function system on $\Lambda^\infty$ as claimed.
\end{proof}

\begin{thm}
\label{th-branching-sys-gives-repres} Let $\Lambda$ be a finite $k$-graph with no sources and suppose that we have a $\Lambda$-semibranching function system on a measure space $(X, \mu)$.  For each $\lambda \in \Lambda$, define $S_\lambda \in B(L^2(X, \mu))$ by
\[S_\lambda \xi(x) = \chi_{R_\lambda}(x)(\Phi_{\tau_\lambda}(\tau^{d(\lambda)}(x)))^{-1/2}\xi(\tau^{d(\lambda)}(x)).\]
Then the operators $\{S_\lambda\}_{\lambda \in \Lambda}$ generate a  representation of $C^*(\Lambda)$.  If $\Lambda$ is aperiodic then this representation is faithful.
\label{thm:gen_repn}
\end{thm}

\begin{proof}

We begin by computing $(S_\lambda)^*$:
\begin{align*}
\l (S_{\lambda})^*\xi, \zeta\r = \l \xi, S_{\lambda}\zeta\r
&= \int_{X} \xi(x) \overline{S_\lambda \zeta(x)} \, d\mu(x) \\
&= \int_{R_\lambda} \xi(x) \Phi_{\tau_\lambda}(\tau^{d(\lambda)}(x))^{-1/2} \overline{\zeta(\tau^{d(\lambda)}(x))} \, d\mu(x) \\
&= \int_{D_\lambda} \xi(\tau_\lambda y) \Phi_{\tau_\lambda}(y)^{1/2}\overline{\zeta(y)}\, d\mu(y).
\end{align*}
So the adjoint of $S_\lambda$ is given by
\[
(S_\lambda^*\xi)(x) = \chi_{D_\lambda}(x) (\Phi_{\tau_\lambda}(x))^{1/2} \xi(\tau_\lambda x).
\]
A straightforward computation, using the fact that $\tau_\lambda \circ \tau^{d(\lambda)} = id$ a.e., will show that $S_\lambda S_\lambda^* S_\lambda = S_\lambda$ as operators on $L^2(X, \mu)$.
Similarly, the hypothesis that $\tau_v = \tau^{d(v)} = id|_{D_v}$  for $v\in\Lambda^0$ implies that
\[S_v\xi(x) = \chi_{D_v}(x) \xi(x),\]
so $S_v$ is a projection for all $v \in \Lambda^0$.

The fact that the projections $S_v$ are mutually orthogonal follows from the hypothesis that $\{\tau_v: v \in \Lambda^0\}$ forms a $\Lambda$-semibranching function system, and consequently that $\mu(D_v \cap D_w) = 0$ if $w \not= v \in \Lambda^0$.  To be precise, fix $v, w\in\Lambda^0$; then compute
\[
S_v S_w \xi(x)=\begin{cases} (S_w\xi)(x)\;\;\text{if $x \in D_v$}\\ 0\quad\quad\quad\quad\text{otherwise}\end{cases}
=\begin{cases} \xi(x)\;\;\text{if $v=w$}\\ 0\quad\quad\text{otherwise}\end{cases}
\]
Thus if $v\ne w$, then $S_v S_w=0$. So $\{S_v: v \in \Lambda\z\}$ consists of mutually orthogonal projections, which gives (CK1).

To check (CK2), fix $\lambda,\nu\in\Lambda$ and compute
\[\begin{split}
S_\lambda S_\nu \xi(x) &= \begin{cases} (\Phi_{\tau_\lambda}(\tau^{d(\lambda)}(x)))^{-1/2} (S_\nu \xi)(\tau^{d(\lambda)}x) \quad\text{if $x\in R_\lambda$,}\\ 0\Quad\qqquad\text{otherwise}\end{cases}\\
&=\begin{cases} (\Phi_{\tau_\lambda}(\tau^{d(\lambda)}(x)))^{-1/2}(\Phi_{\tau_\nu}(\tau^{d(\nu)}(\tau^{d(\lambda)}(x))))^{-1/2}\xi(\tau^{d(\nu)+d(\lambda)}(x))\\
\Quad\quad\quad\text{if $x \in R_\lambda$, $\tau^{d(\lambda)}x \in R_\nu$,}\\
0\Quad\quad\;\;\text{otherwise}\end{cases}
\end{split}\]

Note that $\tau^{d(\lambda)}x \in D_\lambda = D_{s(\lambda)} = R_{s(\lambda)}$, and that $R_\nu \subseteq D_{r(\nu)} = R_{r(\nu)}$ by Condition (c) of Definition \ref{def-lambda-brach-system}.  Therefore, if $S_\lambda S_\nu \xi(x)$ is to be nonzero, we must have  $\tau^{d(\lambda)}x \in R_{s(\lambda)} \cap R_{r(\nu)}$, so this intersection is nonempty.  In order for $S_\lambda S_\nu \xi$ to be a nonzero element of $L^2(X, \mu)$, then, we must have $\mu(R_{s(\lambda)} \cap R_{r(\nu)}) \not= 0$.  Condition (a) of Definition \ref{def-1-brach-system} then tells us that we must have $s(\lambda) = r(\nu).$  Thus, as we check via the computations that follow that $S_\lambda S_\nu = S_{\lambda \nu}$, we will assume that  $s(\lambda) = r(\nu)$.

Note that if $x\in R_\lambda$, then $ x = \tau_\lambda y$ for some $y \in D_\lambda$, and if $\tau^{d(\lambda)}x = y \in R_\nu $, then $y = \tau_\nu (z)$ for some $z \in D_\nu$.  Since $s(\lambda) = r(\nu)$, we can now use Definition \ref{def-lambda-brach-system}(c) to conclude that
\[x = \tau_\lambda \tau_\nu(z) = \tau_{\lambda \nu} z.\]

We claim that if $x \in R_{\lambda \nu}$ we have:
\begin{equation}\label{eq:claim}
\Phi_{\tau_\lambda}(\tau^{d(\lambda)}(x))\Phi_{\tau_\nu}(\tau^{d(\nu)}(\tau^{d(\lambda)}(x)))
=\Phi_{\tau_{\lambda\nu}}(\tau^{d(\lambda\nu)}(x)).
\end{equation}
To see this, we compute that for a.e. $x \in R_\lambda$,
\[
\Phi_{\tau_\lambda}(\tau^{d(\lambda)}(x))=\frac{d(\mu\circ\tau_\lambda)(\tau^{d(\lambda)}(x))}{d\mu(\tau^{d(\lambda)}(x))}
=\frac{d\mu(x)}{d\mu(\tau^{d(\lambda)}(x))},
\]
since  $\tau_\lambda \tau^{d(\lambda)} = id|_{R_\lambda}$ a.e.
Similarly, if $\tau^{d(\lambda)}(x) \in R_\nu$,
\[\Phi_{\tau_\nu}(\tau^{d(\nu)}(\tau^{d(\lambda)}(x))) = \frac{d(\mu\circ \tau_\nu) (\tau^{d(\nu)}( \tau^{d(\lambda)}x))}{d\mu(\tau^{d(\nu)}(\tau^{d(\lambda)}(x)))}  = \frac{d\mu(\tau^{d(\lambda)}(x))}{d\mu(\tau^{d(\nu) + d(\lambda)}(x))}.\]
Thus, for a.e. $x$ such that $x \in R_\lambda, \tau^{d(\lambda)}(x) \in R_\nu$,
\begin{align*}
\Phi_{\tau_\lambda}(\tau^{d(\lambda)}(x))\Phi_{\tau_\nu}(\tau^{d(\nu)}(\tau^{d(\lambda)}(x)))
&= \frac{d\mu(x)}{d\mu(\tau^{d(\lambda)}(x))} \frac{d\mu(\tau^{d(\lambda)}(x))}{d\mu(\tau^{d(\nu) + d(\lambda)}(x))} \\
&=  \frac{d\mu(x)}{d\mu(\tau^{d(\nu) + d(\lambda)}(x))}\\
&= \frac{d\mu(\tau_{\lambda \nu} (\tau^{d(\lambda \nu)}(x))}{d\mu(\tau^{d(\lambda \nu) }(x))}\\
&=\Phi_{\tau_{\lambda\nu}}(\tau^{d(\lambda\nu)}(x))
\end{align*}
whenever $x \in R_{\lambda\nu} \subseteq R_\lambda$.  It remains to show that
\begin{equation}
\label{eq:R-lambda-nu}
R_{\lambda\nu} = \{x\in R_\lambda: \tau^{d(\lambda)}(x) \in R_\nu\}.
\end{equation}
If $x \in R_\lambda$ and $\tau^{d(\lambda)}(x) \in R_\nu$, we have $\tau^{d(\lambda)}(x) = \tau_\nu y$ for some $y \in D_\nu$, so for almost all such $x$ we have
\[y = \tau^{d(\nu)}\circ \tau^{d(\lambda)}x = \tau^{d(\lambda \nu)}(x) \Rightarrow \tau_{\lambda \nu}(y) = x.\]
In other words, $\{x \in R_\lambda: \tau^{d(\lambda)}(x) \in R_\nu\} \subseteq R_{\lambda \nu}$.

Now, suppose $x \in R_{\lambda \nu} \subseteq R_\lambda$, so $x = \tau_\lambda y = \tau_{\lambda \nu} z$.  Then Condition (d) of Definition \ref{def-lambda-brach-system} implies that
\begin{align*}
z = \tau^{d(\lambda \nu)}(x )&= \tau^{d(\nu) +  d(\lambda)}( x) = \tau^{d(\nu)} \circ \tau^{d(\lambda)}(x) = \tau^{d(\nu)}( y)
\end{align*}
 almost everywhere, so $\tau^{d(\nu)}(y) = z \in D_{\lambda \nu} = D_\nu$ for a.e. $y = \tau^{d(\lambda)}(x)$.  Consequently, for almost all $x \in R_{\lambda \nu} \subseteq R_\lambda,$ we have $\tau^{d(\lambda)}(x) \in R_\nu$, so Equation \eqref{eq:R-lambda-nu} (and thus also Equation \eqref{eq:claim})
is true up to a set of measure 0.

Then the claim implies
\begin{align*}
S_\lambda S_\nu \xi(x) &= \begin{cases} (\Phi_{\tau_{\lambda\nu}}(\tau^{d(\lambda\nu)}(x)))^{-1/2}\xi(\tau^{d(\lambda\nu)}(x))\quad\text{if $x \in R_{\lambda\nu}$}\\
0\Quad\text{otherwise}\end{cases}\\
&=S_{\lambda\nu}\xi(x).
\end{align*}
Thus (CK2) holds.

For (CK3), fix $\lambda\in\Lambda$. Then a straightforward computation implies that
\[
S_\lambda^* S_\lambda \xi(x) =\chi_{D_\lambda}(x)\xi(x)=(S_{s(\lambda)})\xi (x).
\]
This gives (CK3).

Finally, we check (CK4): fix $v\in\Lambda^0$ and $n\in\N^k$. Let $\lambda\in v\Lambda^n$. If $x \not \in R_v = D_v \supseteq R_\lambda$, then $S_\lambda S^*_\lambda \xi(x)=0=S_v\xi(x)$, and hence $(\sum_{\lambda\in v\Lambda^n}S_\lambda S^*_\lambda)\xi(x)=0=S_v\xi(x)$. So suppose that $x \in D_v$. We compute
\begin{align*}
 S_\lambda S_\lambda^* \xi(x) &= \begin{cases} (\Phi_{\tau_\lambda}(\tau^{d(\lambda)}(x)))^{-1/2} ( S_\lambda^* \xi)(\tau^{d(\lambda)}x)\quad \text{if $x \in R_\lambda$,}\\ 0\quad\quad\quad\text{otherwise}\end{cases} \\
 &= \begin{cases} \xi(x)\quad\text{if $x\in R_\lambda$}
 \\ 0\quad\quad\quad\text{otherwise}\end{cases}
\end{align*}
Thus $\sum_{\lambda\in v\Lambda^n}S_\lambda S^*_\lambda\xi(x)= \sum_{\lambda \in v\Lambda^n} \chi_{R_\lambda}(x) \xi(x).$
We claim that
\begin{equation}
\label{eq:ck4-claim}\sum_{\lambda \in v\Lambda^n} \chi_{R_\lambda} = \chi_{D_v}.
\end{equation}

To see this, recall that Condition (a) of Definition \ref{def-1-brach-system}
implies that $X = \cup_{d(\lambda) = n} R_\lambda$, so in particular, there exists a set $S \subseteq \Lambda^n$ such that $D_v \subseteq \cup_{\lambda \in S} R_\lambda$.
Suppose that $\lambda \in S$ but $r(\lambda) \not= v$.  If $x \in R_\lambda$, Condition (c) of Definition \ref{def-lambda-brach-system} tells us that $x \in D_{r(\lambda)}$, so the fact that the sets $\{D_w: w \in \Lambda^0\}$ form a $\Lambda$-semibranching function system with $D_w = R_w$ for all $w \in \Lambda^0$ implies that $\mu(D_v \cap D_{r(\lambda)} )=0$ unless $v = r(\lambda)$.  So, we lose nothing by removing from $S$ any $\lambda$ with  $r(\lambda) \not= v$, allowing us to assume that $S\subseteq v\Lambda^n$. Thus $D_v\subseteq \cup_{\lambda\in v\Lambda^n} R_\lambda$.

Condition (c) of Definition \ref{def-lambda-brach-system} now tells us that $R_\lambda \subseteq D_v \ \forall \ \lambda \in v\Lambda$. (Thus $\cup_{\lambda\in v\Lambda^n} R_\lambda\subseteq D_v$, and hence $D_v=\cup_{\lambda\in v\Lambda^n} R_\lambda$). Thus Equation \eqref{eq:ck4-claim} holds. 

It follows that $ \sum_{\lambda \in v\Lambda^n} S_\lambda S^*_\lambda = S_v$, so (CK4) holds.

Since $\{S_\lambda:\lambda\in\Lambda\}$ is a Cuntz-Krieger $\Lambda$-family on $L^2(X,\mu)$, the universal property of $C^*(\Lambda)$ gives a representation $\pi_S:C^*(\Lambda)\to B(L^2(\Lambda^\infty,\mu))$.  For any $v \in \Lambda^0$, we have $S_v\xi(x)=\chi_{D_v}(x)\xi(x)$, and hence $S_v$ is nonzero for each $v\in\Lambda^0$.  Thus, by Theorem 4.6 of \cite{KP}, $\pi_S$ is faithful whenever $\Lambda$ is aperiodic.
\end{proof}

\begin{cor}
Let $\Lambda$ be a strongly connected, aperiodic, finite $k$-graph.  Then we have a faithful representation of $C^*(\Lambda)$ on $L^2(\Lambda^\infty, M)$.
\end{cor}

\subsection{Examples of $\Lambda$-semibranching function systems}
\label{sec:example}
In this section, we strengthen slightly our working hypothesis that $\Lambda$ is a strongly connected finite $k$-graph.   To be precise, we also assume throughout this section that the vertex matrices $A_1,\ldots,A_k$ are $\{0,1\}$ matrices and that the product  $A:=A_1\cdots A_k$ is also a $\{0,1\}$ matrix.   Under these hypotheses, we can construct $\Lambda$-semibranching function systems (and hence, by Theorem \ref{thm:gen_repn},  representations of $C^*(\Lambda)$)
on $L^2(X)$ for several   fractal subspaces $X$ of $[0,1]^n$.  We describe two such constructions in this section.

The reason for the stronger hypotheses is to be able  to apply Theorem 2.17 of \cite{MP} to $A$.   Although the statement of Theorem 2.17 requires $A$ to be irreducible,  we observe that the proof   only requires $A$ to have  a unique unimodular eigenvector with positive entries.  Our  hypothesis that $\Lambda$ is strongly connected implies that the  Perron-Frobenius eigenvector of $\Lambda$ is the  unique Perron-Frobenius eigenvector for $A$, see Definitions  \ref{def:irreduc-family-def} and
\ref{def:irreduc-family-perr-frob-def}.

Moreover, many examples of strongly connected finite $k$-graphs also satisfy the  hypotheses of this section.  One such example (for $k=2$)
 comes from \cite{paskraebuweaver-family}.
\begin{example}
\label{ex-ledrappier} (Ledrappier's example; \cite{paskraebuweaver-periodic},
\cite{paskraebuweaver-family}) The skeleton of this 2-graph is given on page 1622 of \cite{paskraebuweaver-family}. The vertex matrices are
\[
A_1=\begin{pmatrix}
1&0 &0&1 \\1&0 &0&1 \\ 0&1 &1&0\\ 0&1 &1&0
\end{pmatrix},\ A_2=\begin{pmatrix}
1&0 &1&0 \\1&0 &1&0 \\ 0&1 &0&1\\ 0&1 &0&1
\end{pmatrix}
\]
Note also that
the Perron-Frobenius eigenvalues $\rho(A_1)$,   $\rho(A_2)$, and   $\rho(A_1A_2)$ respectively of $A_1, A_2,$ and $A_1A_2$ are given by
\[
\rho(A_1)= \rho(A_2) =2,\ \rho(A_1A_2)=4.
\]

\end{example}

Let $\Lambda$ be a $k$-graph satisfying our newly strengthened hypotheses.
For our first example of a fractal representation of $C^*(\Lambda)$, we will construct a $\Lambda$-semibranching function system on a Cantor-type fractal subspace $X$ of $[0,1]$ for any $k$-graph $\Lambda$ that satisfies  the conditions specified at the beginning of this section.  We begin by describing the construction of $X$ and an embedding of $\Lambda^\infty$ into $X$.

Observe that to  any infinite path $x \in \Lambda^\infty$, we can associate a unique  sequence of edges
\[f_{11}^x, f_{21}^x, \ldots, f_{k1}^x, f_{12}^x, \ldots, f_{k2}^x,  f^x_{13}, \ldots, \] such that
$ x = f^x_{11} f^x_{21} \cdots f^x_{k1} f^x_{12} \cdots  $
and $d(f^x_{ij}) = e_i$ for all $j$. In other words, the decomposition $x = f^x_{11} f^x_{21} \cdots f^x_{k1} f^x_{12} \cdots$ is a ``rainbow decomposition,'' with edges of color $i$ occurring in the $nk+i$th spot for each $n \in \N$.

By our hypothesis that the vertex matrices of $\Lambda$ are 0-1 matrices, we know that given any $v,w \in \Lambda^0$,  there is at most one edge of shape $i$ with source $v$ and range $w$.  This implies that the string of edges $f^x_{11} f^x_{21} \cdots f^x_{k1} f^x_{12} \cdots$ can be replaced by a unique string of vertices $v^x_{11}, v^x_{21}, \ldots, v^x_{k1}, v^x_{12}, \ldots$, where
\[A_i(v^x_{ij}, v^x_{(i+1)j}) = 1 \ \forall \ j, \ \forall \ 1 \leq i \leq k-1 \text{ and } A_k(v^x_{kj}, v^x_{1(j+1)}) = 1 \ \forall \ j.\]
 Write $(v^x)$ for this sequence of vertices.

Write $N = |\Lambda^0|$.  Henceforth we will assume we have chosen a labeling of $\Lambda^0$ by the integers $0, 1, \ldots, N-1$ and we will often confuse a vertex with its label.

With this notation, we now define a map $\Psi: \Lambda^\infty \to [0,1]$, where we write points in $[0,1]$ in their $N$-adic expansions:
\[\Psi(x) = 0. v^x_{11} v^x_{21} \ldots v^x_{k1} v^x_{12} v^x_{22} \ldots = \sum_{1 \leq i \leq k , j \in \N} \frac{v^x_{ij}}{N^{i+k(j-1)}}.\]
The fact that each path $x \in \Lambda^\infty$ is associated to a unique string of vertices $v^x_{11}, v^x_{21}, \ldots, v^x_{k1}, v^x_{12}, \ldots$ implies that $\Psi$ is injective; the image of $\Psi$ in $[0,1]$ is a fractal-type subspace $X$.

Since $\Psi: \Lambda^\infty \to [0,1]$ is injective, we can use $\Psi$ to transfer the prefixing and coding maps from Proposition \ref{prop:kgraph} to $X = \Psi(\Lambda^\infty)$, obtaining a new family of prefixing and coding maps
\[\tau^m := \Psi \circ \sigma^m \circ \Psi^{-1}, \quad \tau_\lambda := \Psi \circ \sigma_\lambda \circ \Psi^{-1}\]
on $X$.  If we further define a measure $\mu$ on $X$ by $\mu(E) = M(\Psi^{-1}(E)),$
then the maps $\{\tau^m, \tau_\lambda\}$ form a $\Lambda$-semibranching function system on $(X, \mu)$.   It follows then from Theorem \ref{th-branching-sys-gives-repres} that we have a representation of $C^*(\Lambda)$ on $L^2(X, \mu)$.

The following is a generalization of Theorem 2.17 of \cite{MP}.

\begin{thm}
\label{thm-RS-Perron-Frobenius}
Let $\Lambda$ be a  strongly connected finite $k$-graph with $\{0,1\}$ vertex matrices $A_1,\dots, A_k$, such that the product $A := A_1 \cdots A_k$ is also a $\{0,1\}$ matrix.
Let $X  = \Psi(\Lambda^\infty)$ and let $\mu$ be the measure on $X$ given by $\mu(E) = M(\Psi^{-1}(E))$, where $M$ is the measure on $\Lambda^\infty$ given by Equation \ref{eq:measure}.  Then $\mu = \frac{1}{H^s(X)}H^s$, where $H^s$ is the Hausdorff measure on $X$ associated to its Hausdorff dimension $s$. Moreover
\[
 N^{k s } =\rho(A), \  s = \frac1k \frac{\ln \rho(A) }{\ln N }
\]
where $N= |\Lambda^0|$ and $\rho(A)$ denotes the Perron-Frobenius eigenvalue  of $A$.
\label{prop:M=Hausd}
\end{thm}

\begin{rmk}
\label{rem-RS-Perron-Frobenius}
 Theorem \ref{thm-RS-Perron-Frobenius} can be rephrased, equivalently, by saying that the two Hilbert spaces $L^2(\Lambda^\infty, M)$  and $L^2(X, \frac{1}{H^s(X)} H^s)$ are isometric  when the hypotheses of the theorem are satisfied.
\end{rmk}

The proof of Theorem \ref{prop:M=Hausd}
requires a number of preliminary steps.  We begin by showing that $X$ has finite and nonzero Hausdorff measure.

\begin{prop}
Let $\Lambda$ be a  strongly connected finite $k$-graph with $\{0,1\}$ vertex matrices $A_1,...,A_k$, such that $A_1...A_k$ is also a $\{0,1\}$ matrix.
 Write $N=|\Lambda^0|$ and write points $x \in [0,1]$ in their $N$-adic expansions:
 \[ x = 0.x_{11} x_{21} \ldots x_{k1} x_{12} x_{22} \ldots x_{k2} x_{13} \ldots.
 \]  Let
\[ X = \{ x \in [0,1]: \forall \ j, \ \forall \ i < k, A_i(x_{ij}, x_{(i+1)j}) = 1 \text{ and } A_k(x_{kj}, x_{1(j+1)}) = 1\}.\]
Then $X$ has finite, positive Hausdorff dimension.
\label{prop:fin-Hausd-dim}
\end{prop}

To prove Proposition \ref{prop:fin-Hausd-dim}, we will need the following Lemma.

\begin{lemma}
\label{lem-holder}
Let $X$ be as in Proposition \ref{prop:fin-Hausd-dim} and define $\phi: X \to [0,1]$ by
\[\phi(x) = \phi(0.x_{11} x_{21} \ldots x_{k1} x_{12} \ldots x_{k2} x_{13} \ldots ) = 0.x_{11} x_{12} x_{13} \cdots.\]
  Then $\phi$ is 1-H\"older continuous.
\label{lem:holder-cts}
\end{lemma}
\begin{proof}
Recall that $\phi$ is 1-H\"older continuous if $N(\phi) :=\sup_{x\not= y} \frac{|\phi(x) - \phi(y)|}{|x-y|} < \infty.$  Thus, suppose that $x \not= y$.  Define $J = \{j\in \N: x_{1j} \not= y_{1j}\};$
without loss of generality, we may assume $J \not= \emptyset$ (if $J = \emptyset$ then $|\phi(x) - \phi(y)| = 0$).

For each $j \in J$, let $\alpha_j = |x_{1j} - y_{1j}|$.  Observe that $|\phi(x) - \phi(y)| = \sum_{j\in J} \frac{\alpha_j}{N^j}$, whereas
\[|x-y| \geq \sum_{j\in J} \frac{\alpha_j}{N^{(j-1)k + 1}} \geq \frac{\alpha_{j_0}}{N^{(j_0-1)k+1}}\]
for any fixed $j_0 \in J$. Consequently,
\begin{align*}
N(\phi) &= \sum_{j \in J} \frac{ \frac{\alpha_j}{N^j}}{|x-y|} \leq \sum_{j\in J} \frac{ \frac{\alpha_j}{N^j}}{\frac{\alpha_{j_0}}{N^{(j_0-1)k + 1}}} = \frac{N^{(j_0-1)k +1}}{\alpha_{j_0}} \sum_{j\in J} \frac{\alpha_j}{N^j} \leq  \frac{N^{(j_0-1)k +1}}{\alpha_{j_0}} < \infty,
\end{align*}
since $\sum_{j\in J} \frac{\alpha_j}{N^j} \leq \sum_{n\in \N} \frac{N-1}{N^n} = 1.$
\end{proof}

\begin{proof}[Proof of Proposition \ref{prop:fin-Hausd-dim}]
Write $A = A_1 A_2 \cdots A_k$.
Observe that if
\[x = 0.x_{11} x_{21} \ldots x_{k1} x_{12} x_{22} \ldots x_{k2} x_{13} \ldots \in X,\] then the sequence $x_{11}, x_{12} ,x_{13}, \cdots$ satisfies
\[A(x_{1j}, x_{1(j+1)})  = 1 \ \forall \ j \in \N.\]
 Our hypotheses imply that $A$ is a $\{0,1\}$ matrix with a unique positive unimodular eigenvector, so the proof of Theorem 2.17 of \cite{MP} says that if we define $\phi: X \to [0,1]$ by $\phi(x) = 0.x_{11} x_{12} x_{13} \cdots$, as in Lemma \ref{lem:holder-cts}, then $\phi(X)$ is the space $\Lambda_{A} \subseteq [0,1]$ studied in Theorem 2.17 of \cite{MP}.  Thus, Part (3) of this theorem tells us that $\phi(X)$ has finite, positive Hausdorff dimension.

Since $\phi$ is 1-H\"older continuous by Lemma \ref{lem:holder-cts}, Proposition 2.2 of \cite{Shah} tells us that the Hausdorff dimension of $X$ is bounded below by that of $\phi(X)$.  In particular, $X$ has nonzero Hausdorff dimension.

To see that the Hausdorff dimension of $X$ is finite, recall that the Hausdorff dimension of $X$ is
\[\delta_X = \inf\{d \geq 0: C^d(X) = 0\},\]
 where
 \[C^d(X) = \inf \{\sum_{i \in I} r_i^d: \text{ there is a cover $\{U_i\}_{i\in I}$ of $X$ with the radius of $U_i$ being $r_i$} \}.\] Thus, if we can show that there exists $d\geq 0$ such that $C^d(X) = 0$, this will establish that $X$ has finite Hausdorff dimension.

 Fix $d > 1$ and let $\ell \in \N$ be arbitrary.  Consider the collection $I_\ell$ of finite $N$-adic numbers in $[0,1]$, such that every element $y \in I_\ell$ has an $N$-adic expansion of length $\ell$.  We define a cover $\{U_y\}_{y\in I_\ell}$ of $X$ by
 \[ U_y = \{x \in X: \text{ the first } \ell \text{ terms of $x$ are given by $y$}\}.\]
 Notice that there are at most $N^\ell$ nonempty sets $U_y$, and that the diameter of the set $U_y$ is strictly less than $1/N^{\ell}$.  Thus, if $d > 1$,
 \[C^d(X) \leq  \inf_{\ell \in \N} \frac{N^\ell}{N^{d\ell}} = \inf_{\ell} \frac{1}{N^{(d-1)\ell}} = 0.\]
 In other words, $\delta_X \leq 1 < \infty$, so the Hausdorff dimension of $X$ is finite and nonzero.
\end{proof}

Having established the existence of a finite, nonzero Hausdorff measure $H^s$ on $X$, we now show that $\frac{1}{H^s(X)}H^s= \mu$ by using the uniqueness of the  Perron-Frobenius measure $M$ of Definition \ref{def:perron-frobenius}.

\begin{proof}[Proof of Theorem \ref{prop:M=Hausd}]
Recall from Proposition 8.1 of \cite{aHLRS3} that $M$ is the unique Borel probability measure on $\Lambda^\infty$ such that $M(Z(\lambda)) = \rho(\Lambda)^{-d(\lambda)} M(Z(s(\lambda)))$ for all $\lambda \in \Lambda$.  By construction, our rescaled Hausdorff measure $ \frac{1}{H^s(X)} H^s$ is a probability measure. Thus, if we show that  $H^s$ satisfies
\begin{equation}
H^s(\Psi(Z(\lambda))) = \rho(\Lambda)^{-d(\lambda)} H^s(\Psi(Z(s(\lambda)))),
\label{eq:hausd-meas}
\end{equation}
this will establish that $\frac{1}{H^s(X)}H^s = \mu$.

We begin by proving Equation \eqref{eq:hausd-meas} when $d(\lambda) = (m, m, \ldots, m)$ for some $m \in \N$.  In this case, if  we abuse notation by writing
\[
\Psi(\lambda) = 0.y_{11} y_{21} \cdots y_{k1} y_{12} \cdots y_{km} y_{1(m+1)},
\]
we have
\[\Psi(Z(\lambda)) = \{x \in X: x_{ij} = y_{ij} \ \forall \ 1 \leq i \leq k, 1 \leq j \leq m+1\}.\]
Similarly,
\[\Psi(Z(s(\lambda))) = \{x \in X: x_{11} = y_{1(m+1)}\} = \{\left( x - \Psi(\lambda) \right) N^{km} + \frac{y_{11}}{N} : x \in \Psi(Z(\lambda))\}.\]
Thus, the scaling- and translation-invariance of the Hausdorff measure (Theorem \ref{thm:hausd-meas-properties}) implies  that
\begin{equation}
H^s(\Psi(Z(s(\lambda)))) = N^{(km)s} H^s(\Psi(Z(\lambda))).
\label{eq:Hs-squares}
\end{equation}

Recall that for any vertex $v$, we can write $Z(v)$ as a disjoint union $Z(v) = \cup_{\lambda \in v\Lambda^{(1,\ldots, 1)}} Z(\lambda)$.  Moreover, we can identify each such path $\lambda$ with the unique string of vertices $(v^\lambda) = v v_1 v_2 \cdots v_k$ such that the edge $f_i$ from $v_i$ to $v_{i-1}$ has shape $e_i$ and $\lambda = f_1 \cdots f_k$. By abuse of notation, we will also write $\lambda = v v_1 v_2 \cdots v_k$. Now, the uniqueness of the string of vertices $(v^\lambda)$, combined with the fact that the vertex matrices $A_i$ are 0-1 matrices, implies that
\[Z(v) = \cup_{v_i \in \Lambda^0} A_1(v, v_1) A_2(v_1, v_2) \cdots A_k (v_{k-1}, v_k) Z(v v_1 \cdots v_k), \]
and hence
\[\Psi(Z(v)) = \cup_{v_i \in \Lambda^0} A_1 (v, v_1) \cdots A_k(v_{k-1}, v_k) \Psi(Z(v v_1 \cdots v_k)).\]
The fact that this union is disjoint means that
\begin{align*}
H^s(\Psi(Z(v)) & = \sum_{v_i \in \Lambda^0} A_1(v, v_1) \cdots A_k (v_{k-1}, v_k) H^s(\Psi(Z(v v_1 \cdots v_k))) \\
&= \sum_{v_i} A_1(v, v_1) \cdots A_k(v_{k-1}, v_k) N^{-ks} H^s(\Psi(Z(v_k)))\\
&= \sum_{v_k \in \Lambda^0} A_1 A_2 \cdots A_k (v, v_k) N^{-ks} H^s(\Psi(Z(v_k))).\end{align*}
In other words, the vector $(H^s(\Psi(Z(v)))_{v\in \Lambda^0}$ is an eigenvector for the product $A = A_1 \cdots A_k$, with eigenvalue $N^{ks}$.

On the other hand, we know from Proposition 3.1 of \cite{aHLRS3} that $x^\Lambda$ is the unique common  unimodular  Perron-Frobenius eigenvector for the vertex matrices  $A_i$.
Since commuting matrices have the same eigenspaces, and $(H^s(\Psi(Z(v)))_{v\in \Lambda^0}$ is an eigenvector for $A$, it follows that $(H^s(\Psi(Z(v)))_{v\in \Lambda^0}$ is an eigenvector for each vertex matrix $A_i$.
Thus, assuming that $H^s(\Psi(Z(v)))$ is nonzero for all vertices $v \in \Lambda^0$, the uniqueness of $x^\Lambda$ tells us that $(H^s(\Psi(Z(v)))_{v\in \Lambda^0}$ must be a scalar multiple of $x^\Lambda$, and moreover that $N^{ks} = \rho(\Lambda)^{(1, \ldots, 1)}$.  Inserting this into Equation \eqref{eq:Hs-squares} tells us that if $d(\lambda) = (m,\ldots, m)$, then
\[H^s(\Psi(Z(\lambda))) = \rho(\Lambda)^{-(m,\ldots, m)} H^s(\Psi(Z(s(\lambda)))).\]

To see that $H^s(\Psi(Z(v)))$ is nonzero, we observe that (in the notation of Lemma \ref{lem:holder-cts})
\[\phi(\Psi(Z(v))) = \{x \in [0,1]: x = 0.v \ldots\} = \Psi(Z(v)). \]
Let $h$ be the Hausdorff measure on $\phi(X)$ associated to its Hausdorff dimension (which we know is finite by Lemma \ref{lem:holder-cts}); then Theorem 2.17(2) of \cite{MP} tells us that $h(\Psi(Z(v))) =x^\Lambda_v.$ In particular, $h(\Psi(Z(v)))$ is finite and nonzero.  Now, the uniqueness of the nonzero Hausdorff measure of a set (Proposition \ref{2.4})  tells us that
\begin{align*}
x^\Lambda_v & =  h(\phi(\Psi(Z(v))) = h (\Psi(Z(v)))\\
& \Rightarrow H^s(\Psi(Z(v)) = h(\Psi(Z(v))) = x^\Lambda_v,
\end{align*}
and $x^\Lambda_v$ is nonzero for all $v \in \Lambda^0$ by definition.

Having thus established Equation \eqref{eq:hausd-meas} when $d(\lambda) = (m, m, \ldots, m)$, we now proceed to the general case.  Let $\lambda \in \Lambda$ be arbitrary, and write $d(\lambda) = (m_1, \ldots, m_k)$. Let $m = \max\{m_i\}$ and let $n = (m, m, \ldots, m) - d(\mu)$.  Define
\[C_\lambda = \{\nu \in \Lambda: r(\nu) = s(\lambda), d(\nu) = n\}.\]
  In words, $C_\lambda$ consists of the paths $\nu$ such that  the product $ \lambda \nu$ is defined and \[d(\lambda \nu) = d(\nu) + d(\lambda) = (m, m, \ldots, m).\]
Moreover, $Z(\lambda)$ can be written as a disjoint union $Z(\lambda) = \cup_{\nu \in C_\lambda} Z(\lambda\nu)$.
Thus,
\[H^s(\Psi(Z(\lambda))) = \sum_{\nu \in C_\lambda} H^s(\Psi(Z(\lambda \nu))) = \sum_{\nu \in C_\lambda} \rho(\Lambda)^{-(m,\ldots, m)} H^s(\Psi(Z(s(\nu)))).\]

For each path $\nu \in C_\lambda$, write $\nu = f^\nu_1 f^\nu_2 \cdots f^\nu_{|d(\nu)|}$, where $|d(f^\nu_i)|=1 \ \forall \ i$, and we list all the color-1 edges first, then all the color-2 edges, etc.
Notice that for each $\nu \in C_\lambda$, there will be $m-m_1$ edges of color 1, then $m-m_2$ edges of color 2, and so on. Write
\[ s(\lambda) v^\nu_{11} v^\nu_{12} \cdots v^\nu_{1(m-m_1)} v^\nu_{21} \cdots v^\nu_{2(m-m_2)} \cdots v^\nu_{k(m-m_k)}\]
 for the unique sequence of vertices associated to $\nu$ in this decomposition.

Since $(H^s(\Psi(Z(v))))_{v\in\Lambda^0}$ is a scalar multiple of $x^\Lambda$ and $x^\Lambda$ is an eigenvector for each vertex matrix $A_i$, with eigenvalue $r(A_i)$ the Perron-Frobenius eigenvalue of $A_i$, we also have that $(H^s(\Psi(Z(v))))_{v\in\Lambda^0}$ is an eigenvector for each vertex matrix $A_i$, with eigenvalue $r(A_i)$. 
Thus,
\begin{align*}
& H^s(\Psi(Z(s(\lambda)))  = \\
& = \rho(\Lambda)^{-(m, \ldots, m) + d(\lambda)}\sum_{v_{ij} \in \Lambda^0}  A_1(s(\lambda), v_{11}) A_1(v_{11}, v_{12}) \cdots A_1(v_{1(m-m_1-1)}, v_{1(m-m_1)}) \\
& \qquad \qquad \qquad \times A_2 (v_{1(m-m_1)}, v_{21}) A_2(v_{21}, v_{22})\cdots A_2(v_{2(m-m_2-1)}, v_{2(m-m_2)}) \cdots \\
& \qquad \qquad \qquad \times \cdots A_k (v_{k(m-m_k-1)}, v_{k(m-m_k)}) H^s(\Psi(Z(v_{k(m-m_k)})))\\
&= \sum_{\nu \in C_\lambda} \rho(\Lambda)^{-d(\nu)}  H^s(\Psi(Z(s(\nu)))),
\end{align*}
since $\nu \in C_\lambda$ precisely when $\nu$ is associated to  a string of vertices
\[v_{11} v_{12} \cdots v_{1(m-m_1)} \cdots v_{k(m-m_k)}
\]
 such that the massive matrix product above is nonzero.

It now follows that
\begin{align*}
H^s(\Psi(Z(s(\lambda)))) &= \rho(\Lambda)^{-(m,\ldots, m) + d(\lambda)} \sum_{\nu \in C_\lambda} H^s (\Psi(Z(s(\nu)))) \\
&=  \rho(\Lambda)^{-(m,\ldots, m) + d(\lambda)} \rho(\Lambda)^{(m,\ldots, m)} H^s(\Psi(Z(\lambda))) \\
&= \rho(\Lambda)^{d(\lambda)} H^s(\Psi(Z(\lambda))).
\end{align*}
In other words, $H^s$ satisfies the scaling property \eqref{eq:hausd-meas} for all $\lambda \in \Lambda$.  It follows that the Hausdorff measure $H^s$ and the measure $\mu$ are multiples of each other, as claimed.
\end{proof}

By a similar construction, one can also produce a $\Lambda$-semibranching function system on a Sierpinski-type fractal set in $[0,1]^{k+1}$, in analogy with the construction in Section 2.6 of \cite{MP}.  Thus, we also obtain a  representation of $C^*(\Lambda)$ on a fractal subspace of $\R^{k+1}$.

 As a corollary of the above proof, we obtain that in many cases, $M$ is the unique measure on $\Lambda^\infty$ that admits a $\Lambda$-semibranching function system.
\begin{cor}
\label{cor:-constant-RN}
  Let $\Lambda$ be a strongly connected finite $k$-graph
  , and suppose that we
have a $\Lambda$-semibranching function system
on the probability space $(\Lambda^{\infty},\mu)$ for some measure $\mu$ on $\Lambda^\infty$.  Suppose  that there exists  $C\in (0,\infty)^k$ such that  the Radon-Nikodym derivative is given by $\Phi_{\tau_\lambda}=C^{-d(\lambda)}$ for all  $\lambda\in\Lambda$. Then $\mu = M$, where $M$ is the Perron-Frobenius measure.
\end{cor}

\begin{proof} As in the proof of Theorem \ref{prop:M=Hausd}, the uniqueness of the measure $M$ means that it will suffice to show that $\mu$ satisfies Equation \eqref{eq:hausd-meas} with $\mu$ replacing $H^s \circ \Psi$. Our hypothesis that the Radon-Nikodym derivative $\Phi_{\tau_\lambda}$ is given by $C^{-d(\lambda)}$ implies that
\[\mu(Z(\lambda)) = C^{- d(\lambda)} \mu(Z(s(\lambda))).\]
One now uses this fact to observe that $(\mu(Z(v)))_{v \in \Lambda^0}$ is an eigenvector for each of the vertex matrices $A_i$ of $\Lambda$, with eigenvalue $C^{e_i}$.  Then the uniqueness of the Perron-Frobenius eigenvector implies that $(\mu(Z(v)))_{v \in \Lambda^0} = x^\Lambda$, and that $C= \rho(\Lambda)$.  As indicated above, the uniqueness of $M$ now implies that $M = \mu$.
\end{proof}

\subsection{KMS states associated to $\Lambda$-semibranching function systems}
\label{sec:kms}

In this section we show how to obtain KMS states on $C^*(\Lambda)$ from certain $\Lambda$-semibranching function systems. 
As we show in Theorem \ref{thm:KMS} below, we obtain KMS states on $C^*(\Lambda)$ whenever the Radon-Nikodym derivative $\Phi_{\tau_\lambda}$ of the $\Lambda$-semibranching function system has a special form and the dynamics on $C^*(\Lambda)$ are given in terms of the Radon-Nikodym derivative. 
 We apply Theorem~\ref{thm:KMS} to our two main examples discussed in Section~\ref{sec:example}. In Corollary~\ref{cor:KMS}, we show that the measure space $(\Lambda^\infty, M)$ given in \eqref{eq:measure} with the preferred dynamics $\alpha^r$ defines a KMS state at the inverse temperature $\beta=1$. In Corollary~\ref{cor-constRN-KMS}, we show that the Cantor-type fractal measure space $(X,\mu)$ associated to the Hausdorff measure defines a KMS state at the inverse temperature $\beta=s$, where $s$ is the Hausdorff dimension. Moreover we show in Corollary~\ref{cor-constRN-KMS} how to apply the ideas of Theorem~\ref{thm:KMS} to obtain KMS states associated to a semibranching function system.  Note that our KMS states provide a concrete realization of certain of the KMS states whose existence was established in \cite{aHLRS2, aHLRS3, aHKR}.

\begin{thm}\label{thm:KMS}
Let $\Lambda$ be a finite $k$-graph 
with no sources, and suppose that we have a $\Lambda$-semibranching function system on the probability space $(X,\mu)$.  Suppose  that there exists  $\omega \in (0,\infty)$ and $C\in (0,\infty)^k$ such that  the Radon-Nikodym derivative is given by $\Phi_{\tau_\lambda}=C^{-\omega d(\lambda)}$ for all  $\lambda\in\Lambda$. Let $\alpha$ be the dynamics on $C^*(\Lambda)$ given by
\[
\alpha_t=\gamma_{C^{ it}},
\]
where $\gamma$ is the gauge action on $C^*(\Lambda)$. Then the measure $\mu$ on $X$ defines a KMS state $\phi$ of the system $(C^*(\Lambda),\alpha)$ at the inverse temperature $\beta=\omega$ by
\begin{equation}\label{eq:state_M}
\phi(s_\lambda s^*_\nu)=\delta_{\lambda,\nu}\, \mu(R_\lambda).
\end{equation}.
\end{thm}

\begin{proof}

Since $\mu$ is a probability measure, $\phi$ is positive and $\phi(1)=1$. Thus $\phi$ is a state on $C^*(\Lambda)$.

 To see that $\phi$ is a KMS state, first we compute, for $\lambda\in\Lambda$
\begin{equation}\label{eq:KMS_measure}
\begin{split}
\mu(R_\lambda)&=\int_{D_\lambda}\frac{d(\mu\circ \tau_\lambda)}{d\mu}\, d\mu=\int_{D_\lambda}\Phi_{\tau_\lambda}\, d\mu =\Phi_{\tau_\lambda}\mu(D_\lambda)=C^{-\omega d(\lambda)}\mu(R_{s(\lambda)})
\end{split}
\end{equation}
So $\phi$ satisfies the following:
\[
\phi(s_\lambda s^*_\nu)=\delta_{\lambda,\nu}C^{-\omega d(\lambda)}\phi(s_{s(\lambda)})\quad\text{for all $\lambda,\nu\in\Lambda$}.
\]
Then Proposition~3.1(b) of \cite{aHLRS2} implies that $\phi$ is a KMS state of $(C^*(\Lambda),\alpha)$ at the inverse temperature $\beta=\omega$.
\end{proof}

 Note that  by Corollary \ref{cor:-constant-RN}, every $\Lambda$-semibranching system on $\Lambda^{\infty}$ with constant Radon-Nikodym derivative of the  form specified in Theorem \ref{thm:KMS}
 is endowed with the Perron-Frobenius measure $M$, hence if $X = \Lambda^\infty$, Theorem \ref{thm:KMS} specializes to the $\Lambda$-semibranching system  of Proposition \ref{prop-main-exampl}.

The dynamics defined on $C^*(\Lambda)$ in Theorem~\ref{thm:KMS} is very similar to the preferred dynamics $\alpha^r$ on $C^*(\Lambda)$, where $r=\ln\rho(\Lambda)$. (See Section~\ref{subsection:background_Dyn}). Thus Proposition~\ref{prop-main-exampl} with the preferred dynamics gives an alternate proof of the first statement of Theorem 11.1 of \cite{aHLRS3}.

\begin{cor}\label{cor:KMS}(\cite{aHLRS3} Theorem 11.1)
Let $\Lambda$ be a strongly connected finite $k$-graph. Let $\alpha^r$ be the preferred dynamics given by
\[
\alpha^r_t=\gamma_{\rho(\Lambda)^{it}}.
\]
 Then there is a unique 
 KMS state $\phi$ of the system $(C^*(\Lambda),\alpha^r)$ at the inverse temperature  $\beta=1$, given by
 \[
\phi(s_\mu s^*_\nu)=\delta_{\mu,\nu}M(Z(\mu))
\]
\end{cor}

\begin{proof}
Proposition~\ref{prop-main-exampl} shows that the measure space $(\Lambda^\infty, M)$ with the prefixing maps $\{\sigma_\lambda:Z(s(\lambda))\to Z(\lambda)\}$ and the coding maps $\{\sigma^m:\Lambda^\infty\to \Lambda^\infty\}$ forms a $\Lambda$-semibranching function system. Also it shows that $\Phi_{\sigma_\lambda}=\rho(\Lambda)^{-d(\lambda)}$ for $\lambda\in\Lambda$. Since $\rho(\Lambda)\in (0,\infty)^k$, Theorem~\ref{thm:KMS} implies that the state $\phi$ on $C^*(\Lambda)$ defined by
\[
\phi(s_\mu s^*_\nu)=\delta_{\mu,\nu}M(Z(\mu))
\]
is a KMS$_1$ state of $(C^*(\Lambda),\alpha^r)$.

The uniqueness of $\phi$ follows from Theorem 11.1 of \cite{aHLRS3}.
\end{proof}

 We   now apply Theorem~\ref{thm:KMS} to  our Cantor-type fractal subspace $X$ with the measure $\mu$ associated to the Hausdorff measure given in Theorem~\ref{thm-RS-Perron-Frobenius}.

\begin{cor} \label{cor-temp-eq-haus-measu}
Let $\Lambda$ be a  strongly connected finite $k$-graph with $\{0,1\}$ vertex matrices $A_1,...,A_k$, such that $A_1 \cdots A_k$ is also a $\{0,1\}$ matrix, as  in Section~\ref{sec:example}. Let $X=\Psi(\Lambda^\infty)$ be the Cantor-type fractal subspace of $[0,1]$ with the probability measure $\mu=\frac{1}{H^s(X)}H^s$ given in Theorem~\ref{thm-RS-Perron-Frobenius}, where $H^s$ is the Hausdorff measure on $X$ associated to its Hausdorff dimension $s$. Define the dynamics $\alpha'$ on $C^*(\Lambda)$ by
\[
\alpha'_t=\gamma_{\rho(\Lambda)^{\frac{it}{s}}},
\]
where $\gamma$ is the gauge action on $C^*(\Lambda)$. Then the measure $\mu$ on $X$ defines a KMS state of the system $(C^*(\Lambda),\alpha')$ at the inverse temperature $\beta=s$.
\end{cor}

\begin{proof}
As shown in Theorem~\ref{thm-RS-Perron-Frobenius},  the map $\Psi: \Lambda^\infty \to X$ is an isometry with respect to the measures $\mu, M$.   Thus, since we know from Proposition \ref{prop-main-exampl} that the Radon-Nikodym derivatives on $(\Lambda^\infty, M)$ are given by $\Phi_{\sigma_\lambda} = \rho(\Lambda)^{-d(\lambda)}$, it follows that  the Radon-Nikodym derivatives  on $(X, \mu)$ are also of the form
\[
\Phi_{\tau_\lambda}=\rho(\Lambda)^{-d(\lambda)}\quad\text{for all $\lambda\in\Lambda$.}
\]
Then Corollary~\ref{cor:KMS} implies that the formula given by
\[
\phi(s_\mu s^*_\nu)=\delta_{\mu,\nu}M(R_\mu)
\]
 defines a KMS state $\phi$ of $(C^*(\Lambda),\alpha^r)$ at the inverse temperature $\beta=1$, where $\alpha^r$ is the preferred dynamics. Since
\[
\alpha'_t=\alpha^r_{ts^{-1}},
\]
Lemma~2.1 of \cite{aHKR} implies that $\phi$ is a KMS state of $(C^*(\Lambda),\alpha')$ at the inverse temperature $\beta=s$, which gives the desired result.
\end{proof}

We can also use the  idea of Theorem \ref{thm:KMS} to construct KMS states associated to semibranching function systems with constant Radon-Nikodym derivatives.

\begin{cor}
\label{cor-constRN-KMS}
Let $I$ be a finite index set.
Let $(X,\mu)$ be a probability measure space that gives a semibranching function system with prefixing maps $\{\sigma_i:D_i\to R_i\}_{i\in I}$ and coding map $\sigma$.  Let  $A=(A_{ij})$ be a $\{0,1\}$-matrix  satisfying
\[\chi_{D_i} = \sum_jA_{ij} \chi_{R_j}.\]
 Let $\mathcal{O}_A$ be the associated Cuntz-Krieger algebra on $L^2(X,\mu)$ as established in Proposition 2.5 of \cite{MP}. Suppose that there exists  $\omega \in (0,\infty)$ and $C\in (0,\infty)$ such that  the Radon-Nikodym derivative is given by $\Phi_{\sigma_i}=C^{-\omega}$ for all  $i\in I$. Let $\alpha$ be the dynamics on $\mathcal{O}_A$ given by
\[
\alpha_t=\gamma_{C^{ it}},
\]
where $\gamma$ is the gauge action on $\mathcal{O}_A$. Then the measure $\mu$ on $X$ defines a KMS state $\phi$ of the system $(\mathcal{O}_A,\alpha)$ at the inverse temperature $\beta=\omega$, by
\begin{equation}\label{eq:state_M}
\phi(S_i S_j^*)=\delta_{i,j}\, \mu(R_i)\quad\text{for $i,j\in I$.}
\end{equation}

\end{cor}

\begin{proof}
Recall that, if $N = |I|$, the Cuntz-Krieger algebra $\mathcal{O}_A$ is the universal $C^*$-algebra generated by $N$ partial isometries $S_i$ such that
\[S_i^* S_i = \sum_j A_{ij} S_j S_j^* \text{ and } \sum_i S_i S_i^* = 1.\]
Moreover, Proposition 2.5 of \cite{MP} establishes that, if we define $T_i \in B(L^2(X, \mu))$ by
\[T_i\xi(x) = \chi_{R_i}(x) \left(\Phi_{\sigma_i}(\sigma(x))\right)^{-1/2} \xi(\sigma(x)),\]
then the operators $T_i$ generate a representation of $\mathcal{O}_A$.

%
Since $\mu$ is a probability measure, $\phi$ is positive and $\phi(1)=1$. Thus $\phi$ is a state on $\mathcal{O}_A$.

For each $i\in I$, we compute
\begin{equation}\label{eq:cor_KMS2}
\mu(R_i)=\int_{D_i}\Phi_{\sigma_i}\,d\mu=C^{-\omega}\mu(D_i).
\end{equation}
To see that $\phi$ is a KMS state of $(\mathcal{O}_A,\alpha)$ at the inverse temperature $\beta=\omega$, it suffices to show that
\[
\phi(S^*_i S_i)=C^\omega\phi(S_iS^*_i)\quad\text{for each $i \in I$.}
\]
Using (2.8) and (2.9) of \cite{MP}, and \eqref{eq:cor_KMS2}, we obtain
\[
\phi(S^*_i S_i)=\sum_{j} A_{ij}\phi(S_jS^*_j)=\sum_{j} A_{ij}\mu(R_j)=\mu(D_i)=C^\omega\mu(R_i)=C^\omega\phi(S_i S^*_i).
\]
Thus $\phi$ is a KMS state of $(\mathcal{O}_A,\alpha)$ at the inverse temperature $\beta=\omega$.

\end{proof}

\begin{rmk}
 When we apply Corollary~\ref{cor-constRN-KMS} to the representation of $\mathcal{O}_A$ on $L^2(\Lambda_A, \mu_A)$, for the Cantor set $\Lambda_A$ described in (2.1) of \cite{MP}, we recover the KMS state of Corollary~2.20 in \cite{MP}.
\end{rmk}

\section{Wavelets on $L^2(\Lambda^\infty, M)$}
\label{sec:wavelets}
We now proceed to construct an orthonormal decomposition of $L^2(\Lambda^\infty, M)$, which we call a \emph{wavelet decomposition}, following Section 3 of \cite{MP}.
Instead of obtaining our wavelets by scaling and translating a basic family of wavelet functions, our wavelet decomposition is constructed by applying (some of) the operators $S_\lambda$ of Theorem \ref{th-branching-sys-gives-repres} to a basic family of functions in $L^2(\Lambda^\infty, M)$.

While we can use the same procedure to obtain a family of orthonormal functions in $L^2(X, \mu)$ whenever we have a $\Lambda$-semibranching function system on $(X, \mu)$, we cannot establish in general that this orthonormal decomposition densely spans $L^2(X, \mu)$ -- we have no analogue of Lemma \ref{lem:squares-span} for general $\Lambda$-semibranching function systems.  Moreover, by Corollary \ref{cor:-constant-RN}, every $\Lambda$-semibranching system  on $\Lambda^{\infty}$ with constant Radon-Nykodim derivative is endowed with the Perron-Frobenius measure. Thus, in this section, we restrict ourselves to the case of $(\Lambda^\infty, M)$.  We also note that our proofs in this section follow the same ideas found in the proof of Theorem 3.2  of \cite{MP}.

For a path $\lambda \in \Lambda$, let $\Theta_\lambda$ denote the characteristic function of $Z(\lambda) \subseteq \Lambda^\infty$.  Recall that $M$ is the unique Borel probability measure on $\Lambda^\infty$ satisfying \eqref{eq:measure}.

\begin{lemma}
\label{lem:squares-span}
Let $\Lambda$ be a strongly connected $k$-graph.
Then the span of the set
\[ S := \{\Theta_\lambda: d(\lambda) = (n, n, \ldots, n) \text{ for some } n \in \N\}\]
is dense in $L^2(\Lambda^\infty, M)$.
\label{dense-squares}
\end{lemma}
\begin{proof}
Let $\mu \in \Lambda$.  We will show that we can write $\Theta_\mu$ as a linear combination of functions from $S$.

Suppose $d(\mu) = (m_1, \ldots, m_k)$.  Let $m = \max\{m_i\}$ and let $n = (m, m, \ldots, m) - d(\mu)$.  Let
\[C_\mu = \{\lambda \in \Lambda: r(\lambda) = s(\mu), d(\lambda) = n\}.\]
  In words, $C_\mu$ consists of the paths that we could append to $\mu$ such that $\mu \lambda \in S$: if $\lambda \in C_\mu$ then the product $\mu \lambda$ is defined and \[d(\mu \lambda) = d(\mu) + d(\lambda) = (m, m, \ldots, m).\]

Observe that, if $\lambda, \lambda' \in C_\mu$ and $\mu \lambda = \mu \lambda'$, the factorization property tells us that $\lambda = \lambda'$.  Similarly, since $d(\mu \lambda) = d(\mu \lambda') = (m, \ldots, m)$, if $x \in Z(\mu \lambda) \cap Z(\mu \lambda')$ then the fact that $x(0, (m, \ldots, m))$ is well defined implies that
\[ x(0, (m, \ldots, m)) = \mu \lambda = \mu \lambda' \Rightarrow \lambda = \lambda'.\]
It follows that if $\lambda, \lambda' \in C_\mu$, then $Z(\mu \lambda) \cap Z(\mu \lambda') = \emptyset$.  Since every infinite path $x\in Z(\mu)$ has a well-defined ``first segment'' of shape $(m, \ldots, m)$ -- namely $x(0, (m, \ldots, m))$ -- every $x \in Z(\mu)$ must live in $Z(\mu \lambda)$ for precisely one $\lambda \in C_\mu$.  Thus, we can write $Z(\mu)$ as a disjoint union,
\[Z(\mu) = \cup_{\lambda \in C_\mu} Z(\mu \lambda).\]

It follows that $\Theta_\mu = \sum_{\lambda \in C_\mu} \Theta_{\mu \lambda}$, so the span of functions in $S$ includes the characteristic functions of cylinder sets.  Since the cylinder sets $Z(\mu)$ form a basis for the topology on $\Lambda^\infty$ with respect to which $M$ is a Borel measure, it follows that the span of $S$ is dense in $L^2(\Lambda^\infty, M )$ as claimed.
\end{proof}

Since the span of the functions in $S$ is dense in $L^2(\Lambda^\infty, M)$, we will show how to decompose $\overline{\sp}\; S$ as
an orthogonal direct sum,
\[\overline{\sp}\; S = \mathcal{V}_{0,\Lambda}\oplus\bigoplus_{j=0}^\infty \mathcal{W}_{j,\Lambda},\]
where we can construct $\mathcal{W}_{j,\Lambda}$ for each $j > 1$ from the functions in $\mathcal{W}_{0,\Lambda}$ and (some of) the operators $S_\lambda$ discussed in Section \ref{sec:sbfs}.  The construction of $\mathcal{W}_{0,\Lambda}$ is similar to that given in Section 3 of \cite{MP} for the case of {a directed graph}.

We begin by setting $\mathcal{V}_{0,\lambda}$ equal to the subspace spanned by the functions $\{\Theta_v: v \in \Lambda\z\}.$ Indeed the functions $\{\Theta_v: v \in \Lambda\z\}$ form an orthogonal set in $L^2(\Lambda^\infty, M),$ whose span includes those functions that are constant on $\Lambda^{\infty}$:
\begin{align*}
\int_{\Lambda^\infty} \Theta_v \overline{\Theta_w} \, dM &= \delta_{v,w} M(Z(v)) \\
&= \delta_{v,w} x^\Lambda_{v},
\end{align*}
and $$\sum_{v\in \Lambda\z}\Theta_v(x)\equiv 1.$$
Thus,  the set $\{\frac{1}{\sqrt{x^\Lambda_v}} \Theta_v: v\in \Lambda\z\}$ is an orthonormal set in $S.$  We define
\[\mathcal{V}_{0,\Lambda}:= \overline{\sp}\{\frac{1}{\sqrt{x^\Lambda_v}} \Theta_v: v\in \Lambda\z\}.\]

To construct $\mathcal{W}_{0,\Lambda}$, let $v \in \Lambda\z$ be arbitrary.  Let
\[D_v = \{\lambda: d(\lambda) = (1, \ldots, 1) \text{ and } r(\lambda) = v\},\]
and write $d_v$ for $|D_v|$ (note that by our hypothesis that $\Lambda$ is a finite $k$-graph we have $d_v < \infty$).

Define an inner product on $\C^{D_v}$ by
\begin{equation}
\label{inner-prod}
\l \vec{v}, \vec{w} \r = \sum_{\lambda \in D_v}  \overline{v_\lambda} w_\lambda \rho(\Lambda)^{(-1, \ldots, -1)} x^\Lambda_{s(\lambda)}\end{equation}
and let $\{ c^{m, v}\}_{m=1}^{d_v - 1}$ be an orthonormal basis for the orthogonal complement of $(1, \ldots, 1) \in \C^{D_v}$ with respect to this inner product.

For each pair $(m,v)$ with $m \leq d_v - 1$ and $v$ a vertex in $\Lambda\z$, define
\[f^{m, v} = \sum_{\lambda \in D_v} c^{m,v}_\lambda \Theta_\lambda.\]
Note that by our definition of the measure  $M$ on $\Lambda^\infty$, since the vectors $c^{m,v}$ are orthogonal to $(1, \ldots, 1)$ in the inner product \eqref{inner-prod}, we have
\begin{align*}
\int_{\Lambda^\infty} f^{m,v} \, dM & = \sum_{\lambda \in D_v} c^{m,v}_\lambda M(Z(\lambda)) \\
&= \sum_{\lambda \in D_v} c^{m,v}_\lambda \rho(\Lambda)^{(-1, \ldots, -1)} x^\Lambda_{s(\lambda)}\\
&= 0
\end{align*}
for each $(m,v)$.  Moreover, the arguments of Lemma \ref{dense-squares} tell us that $\Theta_\lambda \Theta_{\lambda'} = \delta_{\lambda, \lambda'} \Theta_\lambda$ for any $\lambda, \lambda'$ with $d(\lambda) = d(\lambda') = (1, \ldots, 1)$.  Consequently, and if $\lambda\in D_v, \lambda' \in D_{v'}$ for $v \not= v'$, we have  $\Theta_\lambda \Theta_{\lambda'} =0$.  It follows that
\begin{align*}
\int_{\Lambda^\infty} f^{m,v} \overline{f^{m',v'}} \, dM
&= \delta_{v, v'} \sum_{\lambda\in D_v} c^{m,v}_\lambda \overline{c^{m',v}_\lambda} M(Z(\lambda)) \\
&= \delta_{v, v'} \delta_{m,m'}
\end{align*}
since the vectors $\{c^{m,v}\}$ form an orthonormal set with respect to the inner product \eqref{inner-prod}.
Thus, the functions $\{f^{m,v}\}$ are an orthonormal set in $L^2(\Lambda^\infty, M)$.  We define
\[\mathcal{W}_{0,\Lambda} :=\overline{ \sp} \{f^{m,v}: v \in \Lambda\z, 1 \leq m \leq d_v\}.\]

Note that $V_0$ is orthogonal to $\mathcal{W}_{0,\Lambda}$.  To see this, let $ g \in V_0$ be arbitrary, so $g = \sum_{v\in \Lambda\z} g_v \Theta_v$ with $g_v \in \C$ for all $v$. Then
\begin{align*}
\int_{\Lambda^\infty} \overline{f^{m,v'}(x)} g(x) \, dM &= \delta_{v', v} g_v \sum_\lambda c^{m,v}_\lambda M(Z(\lambda)) \\
&= 0,
\end{align*}
since $\sum_\lambda c^{m,v}_\lambda M(Z(\lambda)) = 0$ for all fixed $v, m$.  Thus, $g$ is orthogonal to every basis element $f^{m,v}$ of $\mathcal{W}_{0,\Lambda}$.

The basis $\{f^{m,v}:\; v \in \Lambda\z, 1 \leq m \leq d_v\}$ for $\mathcal{W}_{0,\Lambda}$ is the analogue for $k$-graphs of the  {\it graph wavelets} of \cite{MP}.
As the following Theorem shows, by shifting these functions using the operators $S_\lambda$ of Theorem \ref{th-branching-sys-gives-repres}, we obtain an orthonormal basis for $L^2(\Lambda^\infty, M)$,  and thus $k$-graph wavelets associated to a separable representation of $C^*(\Lambda)$.

\begin{thm}
\label{Wavelets-Theo}  Let $\Lambda$ be a strongly connected finite $k$-graph.
For each fixed $j \in \N_+$ and $v\in\Lambda^0,$ let
\[C_{j,v}:= \{\lambda \in \Lambda: s(\lambda)=v, d(\lambda) = (j,j, \ldots, j)\},\]
 and let $S_\lambda$ be the operator on $L^2(\Lambda^\infty, M)$ described in Theorem \ref{th-branching-sys-gives-repres}.
Then
\[\{S_\lambda f^{m,v}: v \in \Lambda^0, \;\lambda \in C_{j,v},\; 1 \leq m \leq d_v\}\]
is an orthonormal set, and moreover, if $\lambda \in C_{j,v},\; \mu \in C_{i,v'}$ for $0<i<j,$ we have
\[ \int_{\Lambda^\infty} S_\lambda f^{m,v} \overline{S_\mu f^{m',v'}} \, dM = 0 \ \forall \ m, m'.\]

It follows that defining
\[{\mathcal W}_{j,\Lambda} := \clsp\{S_\lambda f^{m,v}: v \in \Lambda^0, \lambda \in C_{j,v}, 1 \leq m \leq d_v\},\]
for $j \geq 1$, we obtain an orthonormal decomposition
\[L^2(\Lambda^\infty, M) = \clsp\;S = {\mathcal V_0}\oplus \bigoplus_{j=0}^\infty {\mathcal W}_{j,\Lambda}.\]
\end{thm}

\begin{proof}
We first observe that if $s(\lambda) = v$, then 
\[S_\lambda f^{m,v} = \sum_{\mu \in D_v} c^{m,v}_\mu \rho(\Lambda)^{d(\lambda)/2} \Theta_{\lambda \mu} ,\]
because the Radon-Nikodym derivatives $\Phi_{\sigma_\lambda}$ are constant on $Z(s(\lambda))$ for each $\lambda \in \Lambda$, thanks to Proposition \ref{prop:kgraph}.
In particular, if $d(\lambda) = 0$ then $S_\lambda f^{m,v} = f^{m,v}$.
Thus, if $d(\lambda) = d( \lambda') = (j, \ldots, j)$, the factorization property and the fact that $d(\lambda \mu ) = d(\lambda'\mu') = (j+1, \ldots, j+1)$ for every $\mu \in D_{s(\lambda)}, \mu' \in D_{s(\lambda')}$ implies that
\[\Theta_{\lambda \mu} \Theta_{\lambda' \mu'} = \delta_{\lambda, \lambda'} \delta_{\mu,\mu'} \ \forall \ \mu \in D_{s(\lambda)}, \mu' \in D_{s(\lambda')}.\]
In particular, $S_\lambda f^{m, v} \overline{S_{\lambda'}f^{m',v'}} = 0$ unless $\lambda = \lambda'$ (and hence $v = v'$).  Moreover,
\begin{align*}
\int_{\Lambda^\infty} S_\lambda f^{m,v} \overline{S_\lambda f^{m',v}} \, dM &=
\sum_{\mu \in D_v} c^{m,v}_\mu \overline{c^{m',v}_\mu}\rho(\Lambda)^{d(\lambda)} M(Z(\lambda \mu))\\
&= \sum_{\mu \in D_v} c^{m,v}_\mu \overline{c^{m', v}_\mu}  \rho(\Lambda)^{-d(\mu)} x^{\Lambda}_{s(\mu)}\\
&=   \delta_{m, m'},
\end{align*}
by the definition of the vectors $c^{m,v}_\mu$, since $d(\mu) = (1,1,\ldots, 1)$ for each $\mu \in D_v$.

Now, suppose $\lambda\in C_{1,v}$.  Observe that $S_\lambda f^{m,v} \overline{f^{m',v'}}$ is nonzero only when $v' = r(\lambda)$, and in this case we have
\begin{align*}
\int_{\Lambda^\infty}S_\lambda f^{m,v} \overline{f^{m',v'}} \, dM &=  \sum_{\mu \in D_v} c^{m,v}_\mu \overline{c^{m',v'}_\lambda} \rho(\Lambda)^{d(\lambda)/2} M(Z(\lambda\mu)) \\
&= \overline{c^{m',v'}_\lambda} \rho(\Lambda)^{-d(\lambda)/2} \sum_{\mu \in D_v}c^{m,v}_\mu \rho(\Lambda)^{-d(\mu)} x^\Lambda_{s(\mu)} \\
&= 0.
\end{align*}
Thus, ${\mathcal W}_{0,\Lambda}$ is orthogonal to ${\mathcal W}_{1,\Lambda}$.

In more generality, suppose that $\lambda \in C_{j,v}, \lambda' \in C_{i,v'},\;j>i\geq 1$.  We observe that $S_\lambda f^{m,v} \overline{S_{\lambda'} f^{m', v'}}$ is nonzero only when $\lambda = \lambda' \nu$ with  $\nu \in C_{j-i,v}$, so we have
$$S_\lambda f^{m,v} \overline{S_{\lambda'} f^{m', v'}} =S_{\lambda'}(S_{\nu}f^{m,v})\overline{S_{\lambda'} f^{m', v'}}.$$
Consequently,
\begin{align*}
\int_{\Lambda^\infty} S_\lambda f^{m,v} \overline{S_{\lambda'} f^{m', v'}} \, dM &= \int_{\Lambda^\infty}S_{\lambda'}(S_{\nu}f^{m,v})\overline{S_{\lambda'} f^{m', v'}}\,dM\\
&= \int_{\Lambda^\infty}(S_{\nu}f^{m,v})\overline{S_{\lambda'}^*S_{\lambda'} f^{m', v'}}\,dM\\
&= \int_{\Lambda^\infty}(S_{\nu}f^{m,v})\overline{f^{m', v'}}\,dM\\
&=  \sum_{\mu \in D_v} c^{m,v}_\mu \overline{c^{m',v'}_\nu} \rho(\Lambda)^{d(\nu)/2} M(Z(\nu\mu)) \\
&= \overline{c^{m',v'}_\nu} \rho(\Lambda)^{-d(\nu)/2} \sum_{\mu \in D_v}c^{m,v}_\mu \rho(\Lambda)^{-d(\mu)} x^\Lambda_{s(\mu)} \\
&= 0.
\end{align*}
Thus, the sets ${\mathcal W}_{j,\Lambda}$ are mutually orthogonal as claimed.
\end{proof}

 In the setting of Theorem
\ref{thm-RS-Perron-Frobenius},  where the infinite path space $\Lambda^{\infty}$ embeds as a fractal subset $X$ of $[0,1],$ following \cite{jonsson} and \cite{MP}, the polynomials of degree $m$ ${\mathcal P}_m$ restricted to $\Lambda^{\infty}$ will form a subspace of dimension $m+1$ of $L^2(X, \mu)$.  For every integer $m\geq 0$ we can define the space $V_{0,\Lambda, m}$ to be those functions in $L^2(X,\mu)$ which, when restricted to the cylinder sets $Z(v),\; v\in \Lambda^0,$ are a polynomial of degree $\leq m.$  Such functions can be thought of as generalized $m$-splines.  Then as in Section 3 of \cite{MP}, it is possible to form generalized multiresolution analyses from the space $V_{0,\Lambda,m}$ by using linear algebra and the partial isometries $S_{\lambda}$ for $\lambda\in C_{j,v}.$


\end{document}